\documentclass{birkjour}

\usepackage{amssymb}
\usepackage{hyperref}
\usepackage{srcltx}
\usepackage{color}

 \newtheorem{thm}{Theorem}[section]
 \newtheorem{cor}[thm]{Corollary}
 \newtheorem{lem}[thm]{Lemma}
 \newtheorem{prop}[thm]{Proposition}
 \theoremstyle{definition}
 \newtheorem{defn}[thm]{Definition}
 \theoremstyle{remark}
 
 \newtheorem*{ex}{Example}
 \numberwithin{equation}{section}

\newcommand{\DD}[1]{\mathbin{\frac{\rm d }{{\rm d }#1}}}
\newcommand{\ddt}{\DD t}

\newcommand{\krein}{Kre\u\i n}
\newcommand{\ud}{\,{\mathrm d}}

\newcommand{\spn}{{\rm span}}

\newcommand \R{{\mathbb R}}
\newcommand \C{{\mathbb C}}
\newcommand \D{{\mathbb D}}

\newcommand \rplus{{\R^{+}}}
\newcommand \cplus{{\C^{+}}}

\newcommand \cminus{{\C^{-}}}
\newcommand{\crange}[1]{\overline{\mathrm{ran}}\left(#1\right)}
\newcommand{\ya}{\reflectbox{\rm R}}     
\newcommand{\bmu}{{\boldsymbol \mu}}
\newcommand{\blam}{{\boldsymbol \lambda}}

\newcommand \rminus{{\R^{-}}}

\newcommand \T{{\mathbb T}}

\newcommand \zero{\set{0}}

\newcommand{\Hscr}{\mathcal H}

\newcommand{\Iscr}{\mathcal I}

\newcommand{\Lscr}{\mathcal L}

\newcommand{\Sscr}{\mathcal S}

\newcommand{\Uscr}{\mathcal U}
\newcommand{\Xscr}{\mathcal X}
\newcommand{\Yscr}{\mathcal Y}
\newcommand{\Zscr}{\mathcal Z}

\newcommand{\Wscr}{\mathcal W}

\newcommand{\AB}{{A\& B}}
\newcommand{\CD}{{C\& D}}

\newcommand{\SysNode}{\bbm{\AB \cr \CD}}
\newcommand{\SmallSysNode}{\sbm{\AB \cr \CD}}

\newcommand{\dom}[1]{{{\rm dom}}\left(#1\right)}

\newcommand{\range}[1]{{\rm im}\bigl(#1\bigr)}

\newcommand{\Ker}[1]{{\rm ker}\left(#1\right)}

\newcommand{\Ipd}[2]{\left\langle #1 , #2 \right\rangle}

\newcommand{\Ipdp}[2]{\left( #1 , #2 \right)}

\newcommand{\set}[1]{\left\lbrace #1 \right\rbrace}

\newcommand{\bigmid}{\bigm \vert}

\newcommand{\bi}{\begin{itemize}}
\newcommand{\ei}{\end{itemize}}
\newcommand{\be}{\begin{enumerate}}
\newcommand{\ee}{\end{enumerate}}

\newcommand{\Bfrak}{\mathfrak B}
\newcommand{\Cfrak}{\mathfrak C}
\newcommand{\Dfrak}{\mathfrak D}

\newcommand{\Rfrak}{\mathfrak R}

\newcommand{\Sfrak}{\mathfrak S}
\newcommand{\Ufrak}{\mathfrak U}

\newcommand{\sbm}[1]{\left[\begin{smallmatrix}#1\end{smallmatrix}\right]}

\newcommand{\bbm}[1]{\begin{bmatrix}#1\end{bmatrix}}

\newcommand{\re}[1]{\mathrm{Re}\,#1}

\newcommand{\cspn}{\overline{\rm span}}

\newcommand{\cH}{{\mathcal H}}

\newcommand{\cL}{{\mathcal L}}

\newcommand{\cU}{{\mathcal U}}
\newcommand{\cX}{{\mathcal X}}
\newcommand{\cY}{{\mathcal Y}}

\newcommand{\dA}{\mathbf A}
\newcommand{\dB}{\mathbf B}
\newcommand{\dC}{\mathbf C}
\newcommand{\dD}{\mathbf D}

\newcommand{\res}[1]{\mathrm{res}\left(#1\right)}

\makeatletter

\def\etv{& \hskip-.3em\vrule\hskip-.3em &} 
\def\smalletv{&\vrule&} 
\def\smallcrh{\vrule height0pt depth2\ex@ width0pt
\cr\noalign{\hrule}
\vrule height6.5\ex@ depth0pt width0pt}
\newbox\smallstrutbox
\setbox\smallstrutbox=\hbox{\vrule height6pt depth1.5pt width\z@}
\def\smallstrut{\relax\ifmmode\copy\smallstrutbox\else\unhcopy\smallstrutbox\fi}

\newenvironment{sysmatrix}{
\let\|=\etv
\hskip \arraycolsep
\begin{matrix}}
{\end{matrix}
\hskip \arraycolsep
}       

\newenvironment{smallsysmatrix}{\null\,\vcenter\bgroup
\let\|=\smalletv

\def\\{\smallstrut\math@cr}
\restore@math@cr\default@tag
\baselineskip\z@skip \lineskip\z@skip \lineskiplimit\lineskip
\ialign\bgroup\hfil$\m@th\scriptstyle##$\hfil&&\thickspace\hfil
$\m@th\scriptstyle##$\hfil\crcr
\crcr\noalign{\vskip -.3\ex@}%
}{\crcr\noalign{\vskip -.2\ex@}%
\crcr\egroup\egroup\,%
}

\makeatother

\begin{document}

%
%
%
%
%
%
%
%
%

\title[A conservative de Branges-Rovnyak functional model on $\mathbb C^+$]{A conservative de Branges-Rovnyak func- \linebreak tional model for operator Schur functions on $\mathbb C^+$}

\author{Joseph A.\ Ball}

\address{%
Department of Mathematics\\
Virginia Tech\\
Blacksburg, VA 24061\\
USA}

\email{joball@math.vt.edu}

\author{Mikael Kurula}
\address{\AA bo Akademi Mathematics\br
Domkyrkotorget 1\br
FIN-20500 \AA bo\br
Finland}
\email{mkurula@abo.fi}


\author{Olof J.\ Staffans}
\address{\AA bo Akademi Mathematics\br
Domkyrkotorget 1\br
FIN-20500 \AA bo\br
Finland}
\email{staffans@abo.fi}



\begin{abstract}
We present a solution of the operator-valued Schur-function realization problem on the right-half plane by developing the corresponding de Branges-Rovnyak canonical conservative simple functional model. This model corresponds to the closely connected unitary model in the disk setting, but we work the theory out directly in the right-half plane, which allows us to exhibit structure which is absent in the disk case. A main feature of the study is that the connecting operator is unbounded, and so we need to make use of the theory of well-posed continuous-time systems. In order to strengthen the classical uniqueness result (which states uniqueness up to unitary similarity), we introduce non-invertible intertwinements of system nodes.
\end{abstract}

\maketitle
\tableofcontents


\section{Introduction}

The classical unitary realization result of de Branges and Rovnyak for Schur functions on the complex unit disk is the following: Let $\Uscr$ and $\Yscr$ be separable Hilbert spaces and let $\phi$ be an operator Schur function on $\D$, i.e., $\phi$ is analytic with $\phi(z)\in\Lscr(\Uscr;\Yscr)$ a contraction for all $z\in\D$. Then the following kernel function on $\D\times\D$, whose values are bounded linear operators on $\sbm{\Yscr\\\Uscr}$, is positive: 
\begin{equation}\label{eq:deBRkern}
 \mathbf K_s(z,w) :=
  \bbm{\displaystyle \frac{1-\phi(z)\,\phi(w)^*}{1-z\overline{w}}
    & \displaystyle\frac{\phi(z)-\phi(\overline w)}{z-\overline{w}} \\
   \displaystyle\frac{\phi(\overline z)^*-\phi(w)^*}{z-\overline{w}}
    & \displaystyle\frac{1-\phi(\overline z)^*\,\phi(\overline w)}{1-z\overline{w}}},
    \qquad z,w\in\D.
\end{equation}
Denoting its reproducing kernel Hilbert space (RKHS) by $\mathbf H_s$, we obtain that the following operator $\sbm{\mathbf A_s&\mathbf B_s \\ \mathbf C_s&\mathbf Ds}:\sbm{\mathbf H_s\\\Uscr}\to\sbm{\mathbf H_s\\\Yscr}$ is unitary:
\begin{equation}\label{eq:deBRunitary}
\begin{aligned}
  \mathbf A_s \bbm{f\\g} &= z\mapsto\bbm{\displaystyle\frac{f(z) - f(0)}{z} \\ 
    z \,g(z) - \phi(\overline{z})^{*}\, f(0)}, \\
  \mathbf B_s\,u &= z\mapsto\bbm{\displaystyle\frac{\phi(z) - \phi(0)}{z} \\ 
    \big(1 - \phi(\overline{z})^{*} \phi(0)\big)} u,\qquad z\in\D,\\
  \mathbf C_s \bbm{f\\g} &= f(0),\qquad
   \mathbf D_s \, u = \phi(0)\,u.
\end{aligned}
\end{equation}
This is the classical de Branges-Rovnyak closely connected unitary functional model for $\phi$. Indeed, it is unitary and closely connected (the disk version of the concept of \emph{simplicity} defined in Def.\ \ref{def:simple} below), its transfer function $G_s(z)=z\mathbf C_s(1-z\mathbf A_s)^{-1}\mathbf B_s+\mathbf D_s$ coincides with $\phi$ on $\D$, and conversely, every closely connected unitary realization of $\phi$ is obtained from \eqref{eq:deBRunitary} by means of a unitary change of state variable.

In this paper, we develop a version of this result in the right-half-plane setting, which requires that we use well-posed systems theory in continuous time. In particular, the analogue of $\sbm{\mathbf A_s&\mathbf B_s \\ \mathbf C_s&\mathbf D_s}$ in \eqref{eq:deBRunitary} is unbounded. Our main results can be summarized in the following simplified form which is completely analogous to the above disk case if one restricts to $\mu_*=\mu$ and $\lambda_*=\lambda$:

\begin{thm}\label{thm:preliminary}
Let $\varphi$ be an operator Schur function on the complex right-half plane $\cplus$. Then the kernel function
\begin{equation}\label{eq:conskernelexp}
  K_s(\mu,\mu_*,\lambda,\lambda_*):=\bbm{\displaystyle \frac{1-\varphi(\mu)\,\varphi(\lambda)^*}{\mu+\overline{\lambda}}
    & \displaystyle \frac{\varphi(\overline{\lambda_*})-\varphi(\mu)}{\mu-\overline{\lambda_*}} \\
  \displaystyle \frac{\varphi(\lambda)^*-\varphi(\overline{\mu_*})^*}{\mu_*-\overline{\lambda}}
    &\displaystyle \frac{1-\varphi(\overline{\mu_*})^*\,\varphi(\overline{\lambda_*})}{\mu_*+\overline{\lambda_*}}},
\end{equation}
$\mu,\mu_*,\lambda,\lambda_*\in\cplus$, is positive; denote its associated RKHS by $\Hscr_s$. 

The following unbounded linear operator from $\sbm{\Hscr_s\\\Uscr}$ into $\sbm{\Hscr_s\\\Yscr}$ is a simple, conservative system node (definitions in \S I.3):
 \begin{align}\notag
  \SysNode_s  &: \bbm{\bbm{x_1\\x_2}\\u}\mapsto\bbm{\bbm{z_1\\z_2}\\y},\quad\text{where} \\
  y  &:= \lim_{\re\eta\to+\infty} \big(\eta\, x_1(\eta)+\varphi(\eta)\,u\big),\quad\text{and} \label{defy'} \\
  \bbm{z_1(\mu)\\z_2(\mu_*)} &= \bbm{\mu \,x_1(\mu)\\-\mu_*\, x_2(\mu_*)}+\bbm{\varphi(\mu)\\1}u-\bbm{1\\\varphi(\overline{\mu_*})^*}y,\quad \mu,\mu_*\in\cplus,\label{defz'}
\end{align}
with domain consisting of all $\sbm{x\\u}$ for which this makes sense:
\begin{equation}\label{eq:domdef}
\begin{aligned}
  \dom{ \SysNode_s }  =& \left\{  \bbm{\bbm{x_1\\x_2}\\u}\in\bbm{\Hscr_s\\\Uscr}\bigmid \right. \\
  &\qquad\left.\text{\eqref{defy'} exists in}~\Yscr~ \text{and}~\text{\eqref{defz'} lies in}~\Hscr_s\right\}.
\end{aligned}
\end{equation}
The transfer function of $\SmallSysNode_s$ is $\varphi$, and conversely every simple, conservative realization of $\varphi$ is unitarily similar to $\SmallSysNode_s$; see Def.\ \ref{def:intertw} below.
\end{thm}

This paper is a direct continuation of \cite{BKSZ15} published earlier in this journal. We refer to that paper as ``Part I'' and assume that the reader is familiar with it. In Part I, the research is placed in its context and detailed background on passive system nodes is presented. Results from Part I will be referenced using a capital `I'; e.g. Thm.\ I.5.1.3 refers to item 3 of \cite[Thm.\ 5.1]{BKSZ15}. In a certain sense, the conservative model is a coupling of the two semi-conservative models in Part I, but working with the conservative model is easier than working with those in Part I. Indeed, the conservative model has the same structure as its adjoint, and hence it combines all the good properties of the semi-conservative realizations.

Investigations closely related to that reported here have been undertaken before, starting from the work \cite{deBrRo66,deBrRoBook} of de Branges and Rovnyak; see \cite{Brod78} for a nice historic overview of work on the disk case up to that point. For a good monograph on the disk case, see \cite{ADRSBook}. The first results in the right-half-plane setting are in \cite{ArNu96}; here Arov and Nudelman used a linear fractional transformation to reduce the half-plane case to the disk case. Most of the more recent publications on  half-plane functional models also employ this so-called Cayley transformation, but in the present paper we work the details out directly in the half-plane setting, in order to expose detail that is invisible in the disk setting.

Adamjan and Arov \cite{AdAr66b} showed how to embed the Sz.-Nagy-Foia\c s model into a suitably more general version of the Lax-Phillips scattering picture (for the discrete-time setting); much later 
Nikolski-Vasyunin \cite{NiVa86,NiVa89,NiVa98} refined this analysis by doing 
such an embedding also for a Pavlov model and a suitably modified 
version of the de Branges-Rovnyak model.  Continuous-time versions of this analysis are also of interest, and we plan to investigate this in a forthcoming publication.

In \cite{BaSt06} an (implicit) ``lurking isometry'' argument and Cayley transformations were used to obtain the existence of a conservative realization for any operator Schur function on the disk or the right-half plane. The realization that we describe in the present paper is a more explicit alternative to the realization constructed in \cite{BaSt06}. The results of \cite{BaSt06} have been extended to a multi-variable case in \cite{BaKV15}, and we expect also the present results to have natural extensions to various multi-variable settings. 

The continuous-time conservative realization has been studied in the state/signal framework developed by Arov and Staffans, too, in \cite{ArKuStCanon}. Here the central idea is to consider in $H^2(\cplus;\Wscr)$ the graph of the Toeplitz operator $T_\varphi$ with symbol $\varphi$, where $\Wscr$ is a \krein\ space, without assuming any particular partition $\Wscr=\Uscr\dotplus\Yscr$ into an input space $\Uscr$ and output space $\Yscr$. In this setting, the action of the realization is a pure shift on the appropriate state space and projections onto input and output components are avoided, which leads to cleaner formulas and intuitively more transparent results; see \cite{ArKuStCanon,ArSt2Canon,ArStConserv} for details. Again, in the present work the objective is to obtain as explicit formulas as possible for the input/state/output setting.

Finally, we mention that a closely related realization of a Nevanlinna family (corresponding to an impedance-passive setting rather than to the present scattering-passive setting) in terms of a boundary relation has been worked out in \cite{BeHaSn08} (or see \cite{BeHaSn09} for a more elaborate version).

The paper is laid out as follows: In \S\ref{sec:sysnode} we briefly present some additional background on conservative and simple system nodes that is needed in the present paper, with the auxiliary proofs on non-invertible intertwinements postponed to Appendix \ref{app:intertwinement}. The conservative model is introduced in \S\ref{sec:consmodel} after its state space has been constructed in \S\ref{sec:statespace}. In \S\ref{sec:extrapolation}, we present an explicit identification of the extrapolation space and calculate the (unbounded) control operator of the conservative model. The paper is concluded in \S\ref{sec:recovering}, where we exhibit the relationship with the classical de Branges-Rovnyak model \eqref{eq:deBRkern}--\eqref{eq:deBRunitary}.

\section{More on passive system nodes}\label{sec:sysnode}

We sharpen the uniqueness results in Part I and also recall a few additional concepts from continuous-time systems theory that are needed in the present paper. The discussion that follows uses the definitions of a passive system node, its main operator $A$, control operator $B$, observation operator $C$, and transfer function given in \S I.3.
 
We recall from Def.\ I.3.6 that $\SmallSysNode:\sbm{\Xscr\\\Uscr}\supset\dom{\SmallSysNode}\to\sbm{\Xscr\\\Yscr}$ is called \emph{scattering dissipative} if for all $\sbm{x\\u}\in\dom{\SmallSysNode}$:
$$
 2\re\Ipdp{z}{x}_\Xscr \leq 
  \Ipdp{u}{u}_\Uscr - \Ipdp{y}{y}_\Yscr,\qquad \bbm{z\\y}=\SysNode\bbm{x\\u}.
$$
By Def.\ I.3.7 every passive system node is scattering dissipative. By the following rather obvious consequence of \cite[Thm.\ 2.5]{StafMaxDiss}, every passive system node is even \emph{maximal} scattering dissipative:

\begin{lem}\label{lem:passmaxdiss}
A passive system node has no scattering-dissipative proper extension.
\end{lem}

Next, we define the concept of simplicity of a continuous-time system. In the following definition, $A_{-1}$ denotes the unique extension of $A$ to a closed operator on the extrapolation space induced by $A$, and $A^d_{-1}$ denotes the analogous extension of $A^*$ to its extrapolation space -- this latter operator was denoted by $A^*\big|_{\Xscr}$ in Part I.

\begin{defn}\label{def:simple}
A passive system node $\SmallSysNode$ with state space $\Xscr$ is \emph{simple} if
\begin{equation}\label{eq:simpledef}
\begin{aligned}
  &\cspn\{(\lambda-A^d_{-1})^{-1}C^*\gamma+({\lambda_*}-A_{-1})^{-1}B\nu\mid \\
  &\qquad\qquad\qquad\qquad\quad \lambda,\lambda_*\in\cplus,\,\gamma\in\Yscr,\,\nu\in\Uscr\}=\Xscr.
\end{aligned}
\end{equation}
\end{defn}

Comparing simplicity to the notions of controllability and observability in \S I.3, one observes that every controllable and every observable passive system is simple; take either $\gamma=0$ or $\nu=0$ in \eqref{eq:simpledef}.

The equation (I.3.6), which is valid for every system node, plays an important role in the theory of de Branges-Rovnyak models on $\cplus$, e.g., in the proof of Thm.\ I.4.3 (or its further development Thm.\ \ref{thm:sharperint} below). For \emph{conservative} systems, we have the additional equality \eqref{eq:twist} below:

\begin{lem}\label{lem:consdual}
Let $\SmallSysNode$ be a conservative system node. The (in general unbounded) adjoint of $\SmallSysNode$ is the system node
\begin{equation}\label{eq:adjoint}
\begin{aligned}
  \SysNode^* &= \bbm{-\bbm{A \& B }\\ \bbm{0 & 1}}\bbm{\bbm{1&0}\\\bbm{C\&D}}^{-1}, \\
  \dom{\SysNode^*} &= \bbm{\bbm{1&0}\\\bbm{C\&D}}\dom{\SysNode}.
\end{aligned}
\end{equation}
For every $\lambda\in\cplus$ and $\gamma\in y$, 
\begin{equation}\label{eq:twist}
\begin{aligned}
  \bbm{(\overline\lambda-A^d_{-1})^{-1}C^*\gamma\\\varphi(\lambda)^*\gamma} &\in \dom{\SysNode}\qquad \text{and} \\
  \SysNode\bbm{(\overline\lambda-A^d_{-1})^{-1}C^*\gamma\\\varphi(\lambda)^*\gamma} &=
    \bbm{-\overline\lambda(\overline\lambda-A^d_{-1})^{-1}C^*\gamma\\\gamma}.
\end{aligned}
\end{equation}
\end{lem}

Sometimes we write \eqref{eq:adjoint} in the ``time-flow inverse'' form 
\begin{equation}\label{eq:tfi}
  \bbm{z\\y}=\SysNode\bbm{x\\u} \quad\Longleftrightarrow\quad \SysNode^*\bbm{x\\y}=\bbm{-z\\u}.
\end{equation}

\begin{proof}[Proof of lemma \ref{lem:consdual}]
The equality \eqref{eq:adjoint} holds by Thm.\ I.3.12. We have
$$
  \bbm{\overline\lambda(\overline\lambda-A^d_{-1})^{-1}C^*\gamma\\\varphi(\lambda)^*\gamma}=\SysNode^*
    \bbm{(\overline\lambda-A^d_{-1})^{-1}C^*\gamma\\\gamma},
$$
and combining this with \eqref{eq:tfi}, we obtain \eqref{eq:twist}.
\end{proof}

\subsection{The passive input, output and past/future maps}

The exposition and terminology of this section loosely follow \cite[\S\S 5\,--\,6]{ArSt2Canon}. 

\begin{thm}\label{thm:iomaps}
Let $\SmallSysNode$ be a passive system node and denote by $\Hscr_o$ and $\Hscr_c$ the Hilbert spaces with reproducing kernels $K_o$ and $K_c$ in (I.1.17), respectively.

The following \emph{(passive frequency-domain) output map} is a contraction from $\Xscr$ into $\Hscr_o$:
\begin{equation}\label{eq:outputaction}
  \Cfrak x:=\mu\mapsto C(\mu-A)^{-1}x,\quad \mu\in\cplus.
\end{equation}
Moreover, the mapping
\begin{equation}\label{eq:inputaction}
  \Bfrak: e_c(\overline\lambda)^*\nu\mapsto (\lambda-A_{-1})^{-1}B\nu,\qquad \lambda\in\cplus,\, \nu\in\Uscr,
\end{equation}
extends by linearity and operator closure to a contraction mapping $\Hscr_c$ into $\Xscr$.
\end{thm}

The theorem can be seen as a consequence of \cite[Thm.\ 11.1.6]{StafBook}, but the connection requires some explanations, and so we include a proof formulated in the present setup for reading convenience.  At the end of the proof we need the following notation which is familiar from Part I:
$$
	\widetilde\varphi(\mu):=\varphi(\overline\mu)^*,\qquad\mu\in\cplus.
$$
Also, we introduce the notation $T_\varphi$ for the usual Toeplitz operator $T_\varphi$ with symbol $\varphi\in L^\infty(i\R;\Uscr,\Yscr)$:
$$
	T_\varphi:=P_+M_\varphi\big|_{H^2(\cplus;\Uscr)}:H^2(\cplus;\Uscr)\to H^2(\cplus;\Yscr),
$$
where $M_\varphi:L^2(i\R;\Uscr)\to L^2(i\R;\Yscr)$ is the operator multiplying by $\varphi$ and $P_+$ is the orthogonal projection of $L^2(i\R;\Yscr)$ onto $H^2(\cplus;\Yscr)$. In our case, we always have $\varphi\in\Sscr(\cplus;\Uscr,\Yscr)\subset H^\infty(\cplus;\Uscr,\Yscr)$, so that $T_\varphi u=M_\varphi u$ for all $u\in H^2(\cplus;\Uscr,\Yscr)$; indeed, in Part I, we used the somewhat less precise notation $M_\varphi$ rather than $T_\varphi$.

\begin{proof}[Proof of Thm. \ref{thm:iomaps}]
See \cite[\S11.1]{StafBook} for background and more details on this proof.

Let $(u,x,y)$ be a stable classical trajectory of the system node $\SmallSysNode$, i.e., $u\in L^2(\rplus;\Uscr)\cap C(\rplus;\Uscr)$, $x\in C^1(\cplus;\Xscr)$, $y\in C(\rplus;\Yscr)$, and
$$
  \bbm{\dot x(t)\\y(t)}=\SysNode\bbm{x(t)\\u(t)},\quad t\geq0.
$$
By Defs.\ I.3.6--7, we then have for all $t\geq0$ that
$$
  \ddt\|x(t)\|^2=\Ipdp{\dot x(t)}{x(t)}_\Xscr + \Ipdp{x(t)}{\dot x(t)}_\Xscr \leq 
  \Ipdp{u(t)}{u(t)}_\Uscr - \Ipdp{y(t)}{y(t)}_\Yscr,
$$
and integrating this from $0$ to $T\geq0$, we obtain
$$
  -\|x(0)\|^2\leq \|x(T)\|^2-\|x(0)\|^2\leq \int_0^T\|u(t)\|^2\ud t-\int_0^T\|y(t)\|^2\ud t.
$$
Letting $T\to+\infty$, we obtain that $y\in L^2(\rplus;\Yscr)$ and $\|x(0)\|^2\geq \|y\|_{L^2(\rplus;\Yscr)}^2-\|u\|_{L^2(\rplus;\Uscr)}^2$; moreover
$$
\begin{aligned}
	\|x(T)\|^2 &\leq \|x(0)\|^2+\int_0^T\|u(t)\|^2\ud t-\int_0^T\|y(t)\|^2\ud t \\
	&\leq \|x(0)\|^2+\|u(t)\|^2_{L^2(\rplus;\Uscr)},
\end{aligned}
$$
and so $\|x(T)\|^2$ is bounded. Thus we may take Laplace transforms, obtaining that $\|x(0)\|^2\geq \|\widehat y\|_{H^2(\cplus;\Yscr)}^2-\|\widehat u\|_{H^2(\cplus;\Uscr)}^2$ and
$$ 
  \bbm{\widehat x(\mu) \\ \widehat y(\mu)} = \bbm{(\mu-A)^{-1}&(\mu-A_{-1})^{-1}B \\ C(\mu-A)^{-1}&\varphi(\mu)}
    \bbm{x(0)\\\widehat u(\mu)},\qquad \mu\in\cplus.
$$
Hence, $\widehat y=\Cfrak x(0)+T_\varphi \widehat u$ and the operator $\bbm{\Cfrak&T_\varphi}$ is a contraction from 
\begin{equation}\label{eq:classinitset}
  \set{\bbm{x(0)\\\widehat u}\bigmid (u,x,y)~\text{classical stable trajectory of}~\SysNode}
\end{equation}
(as a subset of $\sbm{\Xscr\\H^2(\cplus;\Uscr)}$) into $H^2(\cplus;\Yscr)$. By \cite[4.6.11]{StafBook}, the set \eqref{eq:classinitset} contains $\sbm{\dom{A}\\\widehat{H^1_0(\rplus;\Uscr)}}$, where $H^1_0(\rplus;\Uscr)$ is the first order Sobolev space of $\Uscr$-valued functions $u$ with the additional restriction $u(0)=0$, and $\sbm{\dom{A}\\\widehat{H^1_0(\rplus;\Uscr)}}$ is dense in $\sbm{\Xscr\\H^2(\cplus;\Uscr)}$ because $A$ generates a contraction semigroup on $\Xscr$ and the Laplace transformation is unitary.

Hence, $\bbm{\Cfrak&T_\varphi}\sbm{\Cfrak^\diamond\\T_\varphi^*}\leq 1$ on $H^2(\cplus;\Yscr)$, where $\Cfrak^\diamond$ denotes the adjoint of $\Cfrak$ calculated with respect to the inner product in $H^2(\cplus;\Yscr)$ rather than with respect to the inner product in $\Hscr_o$ (in which case we would have written $\Cfrak^*$). Thus $\Cfrak\Cfrak^\diamond\leq 1-T_\varphi T_\varphi^*$ and by Douglas' lemma there exists some contraction $C:\Xscr\to\Ker{1-T_\varphi T_\varphi^*}^\perp$ such that $\Cfrak=(1-T_\varphi T_\varphi^*)^{1/2}\,C$. This implies that $\Cfrak$ is a contraction from $\Xscr$ into $\Hscr_o$, because for every $x\in\Xscr$:
$$
	\|\Cfrak x\|_{\Hscr_o} =
		\|(1-T_\varphi T_\varphi^*)^{1/2}\, Cx\|_{\Hscr_o}
		=\|Cx\|_{H^2(\cplus;\Yscr)}\leq \|x\|_\Xscr.
$$
Consider now the output map $\Cfrak^d$ of the passive system node $\SmallSysNode^*$. From (I.1.17) it follows that $\Hscr_o$ constructed with $\widetilde\varphi$ is the same as $\Hscr_c$ constructed with $\varphi$. Thus $\Cfrak^d=\Bfrak^*$ is a contraction from $\Xscr$ into $\Hscr_c$, and so $\Bfrak$ is a contraction from $\Hscr_c$ into $\Xscr$.
\end{proof}

From now on we let $\Bfrak$ be the extension of \eqref{eq:inputaction} by linearity and continuity and we call it the \emph{(passive frequency-domain) input map}. Please note that the input and output maps of the dual system $\SmallSysNode^*$ are $\Cfrak^d\in\Lscr(\Xscr;\Hscr_c)$ and $\Bfrak^d\in\Lscr(\Hscr_o;\Xscr)$ given by
\begin{equation}\label{eq:dualio}
\begin{aligned}
	\Bfrak^d &= \Cfrak^* :e_o(\lambda)^*\gamma\mapsto 
		(\overline\lambda-A_{-1}^d)^{-1} C^*\gamma \qquad\text{and}\\
	\Cfrak^d &= \Bfrak^*:x\mapsto \big(\mu\mapsto B^*(\mu-A^*)^{-1}x\big),
\end{aligned}
\end{equation}
respectively; see Prop.\ I.3.10. We shall later make extensive use of the \emph{unobservable subspace} $\Ufrak$ and the \emph{approximately reachable subspace} $\Rfrak$ of $\SmallSysNode$, which are given by
\begin{equation}\label{eq:UnobsReach}
	\Ufrak:=\Ker{\Cfrak}\qquad\text{and}\qquad \Rfrak:=\crange\Bfrak
\end{equation}
(closure in $\Xscr$). Also note that $\Ufrak^\perp$ is the approximately reachable subspace and $\Rfrak^\perp$ the unobservable subspace of the dual system node $\SmallSysNode^*$. Therefore, we denote $\Ufrak^\perp=:\Rfrak^\dagger$ and $\Rfrak^\perp=:\Ufrak^\dagger$.

\begin{defn}\label{def:pfmap}
For a passive system node $\SmallSysNode$, we call the contraction $\Gamma:=\Cfrak\,\Bfrak:\Hscr_c\to\Hscr_o$ the \emph{(frequency domain) past/future-map} of $\SmallSysNode$.
\end{defn}

We first give the action of the past/future map on kernel functions.

\begin{prop}\label{prop:pfmap}
 The transfer function $\varphi$ uniquely determines $\Gamma$ via 
\begin{equation}\label{eq:Gamma}
  \Gamma\, e_c(\overline{\lambda_*})^*\nu = \mu\mapsto\frac{\varphi({\lambda_*})
  	-\varphi(\mu)}{\mu-{\lambda_*}}\nu,\qquad \lambda_*,\mu\in\cplus,\,\nu\in\Uscr.
\end{equation}
The adjoint of $\Gamma$ has the following action on kernel functions:
\begin{equation}\label{eq:GammaAdj}
  \Gamma^*\,e_o(\overline\lambda)^*\gamma 
  	=\mu\mapsto \frac{\widetilde\varphi(\lambda)-\widetilde\varphi(\mu)}
  		{\mu-\lambda}\gamma,
	\quad \lambda\in\cplus,\,\gamma\in\Yscr.
\end{equation}
Denoting the past/future map of $\varphi$ by $\Gamma_\varphi:\Hscr_c\to\Hscr_o$, and similar for $\Gamma_{\widetilde\varphi}:\Hscr_o\to\Hscr_c$, we have
\begin{equation}\label{eq:GammaDual}
  \Gamma_\varphi^*=\Gamma_{\widetilde\varphi}.
\end{equation}
\end{prop}

As a consequence, all passive realizations of the same transfer function have the same past/future map. 

\begin{proof}[Proof of Prop.\ \ref{prop:pfmap}]
The equation \eqref{eq:Gamma} follows from the following computation:
\begin{equation}\label{eq:resdiff}
\begin{aligned}
	\left(\Gamma e_c(\overline{\lambda_*})^*\right)(\mu) &= C(\mu-A)^{-1}({\lambda_*}-A_{-1})^{-1}B \\
	&= \CD\bbm{(\mu-A)^{-1}({\lambda_*}-A_{-1})^{-1}B\\0} \\
    	&= \CD\frac{\bbm{({\lambda_*}-A_{-1})^{-1}B\\1}-\bbm{(\mu-A_{-1})^{-1}B\\1}}{\mu-{\lambda_*}} \\
     	&= \frac{\varphi({\lambda_*})-\varphi(\mu)}{\mu-{\lambda_*}}.
\end{aligned}
\end{equation}
Then \eqref{eq:GammaAdj} follows from (valid for all $\lambda_*,\lambda\in\cplus$, $\nu\in\Uscr$, and $\gamma\in\Yscr$):
\begin{equation}\label{eq:GammaDefCalc}
\begin{aligned}
  \Ipdp{\Gamma\, e_c(\overline{\lambda_*})^*\nu}{e_o(\overline\lambda)^*\gamma}_{\Hscr_o} 
    &= \Ipdp{\frac{\varphi({\lambda_*})-\varphi(\overline\lambda)}{\overline\lambda-{\lambda_*}}\nu}{\gamma}_{\Uscr} \\
  &= \Ipdp{e_c(\overline{\lambda_*})^*\nu}{\frac{\widetilde\varphi(\cdot)-\widetilde\varphi(\lambda)}{\lambda-\cdot}\gamma}_{\Hscr_c}.
\end{aligned}
\end{equation}
Finally, by \eqref{eq:Gamma} and \eqref{eq:GammaAdj}, $\Gamma_\varphi^* \, e_o(\overline\lambda)^*\gamma=\Gamma_{\widetilde\varphi} \, e_o(\overline\lambda)^*\gamma$ for all $\lambda\in\cplus$ and $\gamma\in\Yscr$. Considering linear combinations of kernel functions and extending by continuity, we obtain \eqref{eq:GammaDual}.
\end{proof}

Combining \eqref{eq:GammaAdj} with (I.4.50), we see that 
\begin{equation}\label{eq:tauGamma}
  \tau_{c,\alpha}=C_c(\alpha-A_c)^{-1}=e_c(\alpha)\,\Gamma,\qquad \alpha\in\cplus.
\end{equation}
This operator played a very important role in the more explicit representation of the energy-preserving controllable model $\SmallSysNode_c$ and its extrapolation space in Part I. In the present paper, the operator $\Gamma$ plays an even more crucial role, already in the proof that the reproducing kernel defining the state space of the simple conservative model is positive.

\begin{ex}\label{ex:canonio}
From \eqref{eq:tauGamma} and \eqref{eq:outputaction}, we obtain that the controllable energy-preserving model $\SmallSysNode_c$ in \S I.4 has $\Cfrak_c=\Gamma$. Moreover, comparing \eqref{eq:inputaction} with (I.4.7), it becomes evident that $\Bfrak_c=1$. Similarly, (I.5.8) implies that the observable co-energy-preserving model $\SmallSysNode_o$ has input map $\Bfrak_o=\Gamma$ and output map $\Cfrak_o=1_{\Hscr_o}$.

\end{ex}

Let $e_+(\mu)$ and $e_-(\mu_*)$ denote point-evaluation of functions in $H^2(\cplus)$, at $\mu\in\cplus$, and $H^2(\cminus)$, at $\mu_*\in\cminus$, respectively. The right-hand side of \eqref{eq:Gamma} equals the action of the Hankel operator $\mathrm H_\varphi:=P_{H^2(\cplus;\Yscr)} \,M_\varphi\big|_{H^2(\cminus;\Uscr)}$ on a kernel function of $H^2(\cminus;\Uscr)$, namely
$$
  k_-(\eta,-\overline{\lambda_*})\,\nu:=\displaystyle\frac{-\nu}{\eta+(-{\lambda_*})},\qquad\eta,\,-{\lambda_*}\in\cminus,\, \nu\in\Uscr.
$$
Indeed, for all fixed parameters $\lambda_*\in\cplus$, for all fixed vectors $\nu\in\Uscr$, and for almost all values of the variable $\mu\in i\R$:
\begin{equation}\label{eq:projtrick}
\varphi(\mu)  \frac{-\nu}{\mu -{\lambda_*}} = 
\frac{\varphi({\lambda_*}) - \varphi(\mu)}
{\mu - {\lambda_*}} \nu 
- \frac{ \varphi({\lambda_*})}{\mu - {\lambda_*}}\nu,
\end{equation}
where the first term is in $H^{2}(\cplus;\Yscr)$ and the second term is in 
$H^2(\cminus;\Yscr)$; from here we deduce that
$$
  {\mathrm H}_{\varphi} \,e_-(-\overline{\lambda_*})^*\,\nu = \mu\mapsto
 \frac{\varphi({\lambda_*}) - \varphi(\mu)}{\mu - {\lambda_*}} \nu,
 \qquad\mu\in\cplus,
$$
where $e_-(-\overline{\lambda_*})^*\,\nu =k_-(\cdot,-\overline{\lambda_*})\,\nu$. Then, using the injections $\iota_c:\Hscr_c\to H^2(\cplus;\Uscr)$ and $\iota_o:\Hscr_o\to H^2(\cplus;\Yscr)$, we have for all $\mu,\lambda_*\in\cplus$ that
$$
  e_+(\mu)\,\iota_o\, \Cfrak\,\Bfrak \, \iota_c^* \, e_+(\overline{\lambda_*})^* = e_c(\mu)\, \Gamma\, e_c(\overline{\lambda_*})^* 
  = e_+(\mu) \,\mathrm H_\varphi\, e_-(-\overline{\lambda_*})^*;
$$
see Thm.\ I.2.4. Taking linear combinations and closing, we get
\begin{equation}\label{eq:HankelFact}
  \iota_o\, \Cfrak\,\Bfrak \, \iota_c^*=\mathrm H_\varphi\, \ya,
\end{equation}
with the reflection being $\ya\, :e_+(\overline{\lambda_*})^*\,\nu\mapsto e_-(-\overline{\lambda_*})^*\,\nu$, $\lambda_*\in\cplus$, $\nu\in\Uscr$, extended by linearity and continuity to a unitary operator $H^2(\cplus;\Uscr)\to H^2(\cminus;\Uscr)$. 

The contractions $\Bfrak \, \iota_c^*:H^2(\cplus;\Uscr)\to\Xscr$ and $\iota_o\, \Cfrak:\Xscr\to H^2(\cplus;\Yscr)$ factorizing the Hankel operator are also sometimes referred to as (frequency-domain) input and output maps, but here we refer to $\Bfrak$ and $\Cfrak$ by these names. From a systems-theory point of view, it would perhaps be more natural to take the state space $\Hscr_c$ of the controllable energy-preserving model to be a subspace of $H^2(\cminus;\Uscr)$ rather than a subspace of $H^2(\cplus;\Uscr)$, since one in time domain often considers input signals in past time $\rminus$ rather than in future time $\rplus$. In particular, the reflection is then absent in \eqref{eq:HankelFact}.

\subsection{A generalized uniqueness result}\label{sec:Intertw}

We see that $\Bfrak$ in the previous subsection is precisely the unitary similarity operator $\Delta$ in Thm.\ I.4.3 and $\Cfrak$ is the adjoint of $\Delta$ in Thm.\ I.5.2. We now proceed to obtain improvements on these uniqueness results in Part I. First we need to relax the notion of unitary similarity from Thm.\ I.4.3. Please note that this subsection first considers general system nodes, not only passive ones.

\begin{defn}\label{def:intertw}
Let $\SmallSysNode_0$ and $\SmallSysNode_1$ be two system nodes with state spaces $\Xscr_0$ and $\Xscr_1$, respectively, and the same input spaces $\Uscr$ and output spaces $\Yscr$. Let $E$ map $\Xscr_0$ linearly and boundedly into $\Xscr_1$.

We say that $E$ \emph{intertwines $\SmallSysNode_0$ with $\SmallSysNode_1$} if 
\begin{align}
	\bbm{E&0\\0&1}\dom{\SysNode_0} &\subset \dom{\SysNode_1} \qquad\text{and}  \label{eq:intertwdom} \\
  	\bbm{E&0\\0&1}\SysNode_0 &= \SysNode_1\bbm{E&0\\0&1}\bigg|_{\dom{\SmallSysNode_0}}. \label{eq:intertw} 
\end{align}
If $E$ is a contraction (an isometry), then we call the intertwinement \emph{contractive (isometric)}, and if $E$ is unitary then we say that $\SmallSysNode_0$ and $\SmallSysNode_1$ are \emph{unitarily similar}.
\end{defn}

Alternative characterizations of intertwinement and more detail can be found in Appendix \ref{app:intertwinement}. Paralleling \cite[Thms.\ 8.4 and 9.5]{ArKuStCanon}, we have the following uniqueness result which is stronger than Thms.\ I.4.3 and I.5.2:

\begin{thm}\label{thm:sharperint}
The following statements hold for every passive system node $\SmallSysNode$ with transfer function $\varphi$, input map $\Bfrak$, and output map $\Cfrak$:

\begin{enumerate}
\item The input map $\Bfrak$ intertwines the energy-preserving model $\SmallSysNode_c$ for $\varphi$ contractively with $\SmallSysNode$. This intertwinement is \emph{isometric} if $\SmallSysNode$ is energy preserving. The intertwinement $\Bfrak$ has range dense in $\Xscr$ if and only if $\SmallSysNode$ is controllable. Moreover, $\Bfrak$ is \emph{unitary} if and only if $\SmallSysNode$ is controllable and energy-preserving.

\item The output map $\Cfrak$ intertwines $\SmallSysNode$ contractively with the co-energy-preserving observable model $\SmallSysNode_o$ for $\varphi$. This intertwinement is co-isometric if $\SmallSysNode$ is co-energy preserving and $\Cfrak$ is injective if and only if $\SmallSysNode$ is observable. Furthermore, $\Cfrak$ is unitary if and only if $\SmallSysNode$ is observable and co-energy preserving.

\item In particular, $\Bfrak_o=\Gamma$ intertwines $\SmallSysNode_c$ with $\SmallSysNode_o$. Hence, for all $\sbm{x\\u}\in\dom{\SmallSysNode_c}$ and $\sbm{z\\y}=\SmallSysNode_c\sbm{x\\u}$: \begin{equation}\label{eq:GammaSc}
	(\Gamma z)(\mu) = \mu \cdot (\Gamma x)(\mu) + \varphi(\mu)\,u-y,\qquad \mu\in\cplus.
\end{equation}
\end{enumerate}
\end{thm}

\begin{proof}
We begin with statement one. In Thm.\ I.4.3, $\Delta=\Bfrak$ and the intertwinement part of the proof goes through even if this operator is only continuous. The proof that $\Bfrak$ is isometric if $\SmallSysNode_c$ is energy preserving is also the same as in Thm.\ I.4.3. The connection between controllability and dense range is immediate from \eqref{eq:UnobsReach}. That $\SmallSysNode$ is energy preserving in the unitary case follows from Thm.\ \ref{thm:intertwprops}.3.

We obtain statement two by duality: The input map of $\SmallSysNode^*$ is $\Cfrak^*$ by \eqref{eq:dualio}, and this operator intertwines the energy-preserving model $\SmallSysNode_o^*$ for $\widetilde\varphi(\mu)=\varphi(\overline\mu)^*$ contractively with $\SmallSysNode^*$; see the introduction to \S I.5. Using Lemma \ref{lem:dualinter}, we obtain that $\Cfrak$ intertwines $\SmallSysNode$ with $\SmallSysNode_o$. The rest of the claim is immediate from the definitions of co-isometry and co-energy-preserving system node.

By Ex.\ \ref{ex:canonio}, $\Bfrak_o=\Gamma$, and then assertion one with Def.\ \ref{def:intertw} gives
$$
	\bbm{z\\y}=\SysNode_c\bbm{x\\y}\quad\implies\quad
	\bbm{\Gamma z\\y}=\SysNode_o\bbm{\Gamma x\\y};
$$
finally Thm.\ I.5.1.3 gives \eqref{eq:GammaSc}.
\end{proof}

\section{The state space of the conservative model}\label{sec:statespace}

The first step in the development is to construct a positive $2\times 2$-block kernel function using only $\varphi$, whose reproducing kernel Hilbert space will be the state space of the realization $\SmallSysNode_s$. As in \cite[\S8]{ArStPartIV}, we develop the theory using a four-variable kernel rather than the standard two-variable kernel, hoping to make visible how the observable co-energy-preserving and controllable energy-preserving functional models are combined into the conservative simple model.

We begin with a general result on how a RKHS can arise as the range of a multiplication operator:

\begin{lem}\label{lem:multspace}
Let $\Omega$ be a point set, $\cX$, $\cY$ be 
Hilbert spaces and let $H \colon \Omega \to \cL(\cX, \cY)$  be an 
operator-valued function.  Define a subspace $\cH_{H}$ of the linear space of $\cY$-valued functions on $\Omega$ by
$$
	\cH_{H} := \{ H(\cdot)\, x \colon x \in \cX\}
	\quad\text{with lifted norm}\quad
	\| H(\cdot)\, x \|_{\Hscr_H} = \| P_{\Ker {M_H}^\perp} x \|_\Xscr,
$$
where $M_{H}$ is the multiplication operator $M_H \colon x 
\mapsto H(\cdot)\,x$ from $\cX$ to $\cH_{H}$.  

Then $\cH_{H}$ is a reproducing kernel Hilbert space 
with reproducing kernel $K_{H}(z,w) := H(z) H(w)^{*}$ and $M_H$ maps $\Ker{M_H}^\perp$ unitarily onto $\Hscr_H$.
\end{lem}

\begin{proof}
By Thm.\ I.1.1, $K$ is the reproducing kernel of some uniquely determined Hilbert space of functions. Let $H(\cdot) x \in \cH_{H}$ and let $w \in \Omega$. 
For $y \in \cY$ we note that
$$
\langle H(w)\, x, y \rangle_{\cY} = \langle x, H(w)^{*}y \rangle_{\cX}.
$$
Noting that $\Ker {M_H} = \{ \xi \in \cX \colon H(z)\, \xi = 
0 \text{ for all } z \in \Omega\}$, we get $H(w)^{*}y \in \Ker{M_H}^{\perp}$ and
$\langle x, H(w)^{*}y \rangle_{\cX} = \langle x', H(w)^{*}y \rangle_{\cX}$
where $x' = P_{(\operatorname{Ker} M_{H})^{\perp}} x$.  By 
construction $M_{H}$ is an isometry from $\Ker{M_H}^{\perp}$ onto $\cH_{H}$.  Hence the above calculation 
continues as
$$
\langle H(w)\, x, y \rangle_{\cY} = \langle H(\cdot) \, x', H(\cdot) \,
H(w)^{*}y \rangle_{\cH_{H}} = \langle H(\cdot)\,x, H(\cdot)\, H(w)^{*}y \rangle_{\cH_{H}}
$$
and we conclude that $K(z,w) = H(z) H(w)^{*}$ works as the RK of $\cH_{H}$. 
\end{proof}

Taking $\Omega=\cplus$ and either $H(\mu)=C(\mu-A)^{-1}$ or $H(\mu)=B^*(\mu-A^*)^{-1}$, we get the following interesting special cases (recall Thm.\ \ref{thm:iomaps} and the text around \eqref{eq:dualio}):

\begin{cor}\label{cor:multspace}
Assume that $\SmallSysNode$ is a passive system node with input/state/\newline output spaces $(\Uscr,\Xscr,\Yscr)$, input map $\Bfrak$, and output map $\Cfrak$. 

If $\range{\mathfrak C}$ is equipped with the lifted norm $ \| {\mathfrak C} x \|_{\range\Cfrak} = \| P_{\Rfrak^\dagger}x \|_{\cX}$ then ${\mathfrak C}$ maps $\Rfrak^\dagger$ unitarily onto $\range\Cfrak$ which is a RKHS $\Hscr_\Cfrak$ with reproducing kernel
$$ K_{\mathfrak C}(\mu, \lambda) = C(\mu - A)^{-1} (\overline{\lambda} - A^{d}_{-1})^{-1} 
C^{*},\qquad \mu,\lambda\in\cplus.$$
Similarly, $\Bfrak^*$ is a unitary identification of $\Rfrak$ with $\Hscr_{\Bfrak^*}$, where
$$
	\Hscr_{\Bfrak^*}:=\range{\Bfrak^*}
	\qquad\text{with lifted norm}\qquad
	\|\Bfrak^*x\|_{\range{\Bfrak^*}}=\|P_{\Rfrak}\, x\|_\Xscr
$$
is the RKHS with reproducing kernel
$$ K_{\mathfrak B^*}(\mu, \lambda) = B^*(\mu - A^*)^{-1} (\overline{\lambda} - A_{-1})^{-1} B,\qquad \mu,\lambda\in\cplus.$$
\end{cor}

The following result which draws some inspiration from \cite[Thm.\ 2.1.2]{ADRSBook} determines the kernel function \eqref{eq:conskernelexp} needed to define the state space of the conservative realization.

\begin{prop}\label{prop:kolmogorov}
Let $\SmallSysNode$ be an arbitrary system node with transfer function $\varphi$.
\begin{enumerate}
\item If $\SmallSysNode$ is co-energy preserving then the reproducing kernel $K_o$ defining the state space for the observable model $\SmallSysNode_o$ factorizes as
\begin{equation}\label{eq:coenpreskernel}
  K_o(\mu,\lambda)=\frac{1-\varphi(\mu)\varphi(\lambda)^*}{\mu+\overline{\lambda}}=C(\mu-A)^{-1}(\overline\lambda-A^d_{-1})^{-1}C^*,\quad \mu,\lambda\in\cplus,
\end{equation}
and the output map $\Cfrak$ maps $\Rfrak^\dagger$ unitarily onto $\Hscr_o$.

\item Set as before $\widetilde\varphi(\mu):=\varphi(\overline\mu)^*$, $\mu\in\cplus$. If $\SmallSysNode$ is energy preserving then $K_c$ associated to the controllable model $\SmallSysNode_c$ factorizes as
\begin{equation}\label{eq:enpreskernel}
  K_c(\mu_*,\lambda_*)=\frac{1-\widetilde\varphi(\mu_*)\widetilde\varphi(\lambda_*)^*}{\mu_*+\overline{\lambda_*}}=
    B^*(\mu_*-A^*)^{-1}(\overline{\lambda_*}-A^d_{-1})^{-1}B,
\end{equation}
$\mu_*,\lambda_*\in\cplus$, and the input map $\Bfrak$ maps $\Hscr_c$ unitarily onto $\Rfrak$.

\item Define
\begin{equation}\label{eq:HGdef}
  H(\mu):=C(\mu-A)^{-1},\quad G(\mu_*):=B^*(\mu_*-A^*)^{-1},\quad \mu,\mu_*\in\cplus.
\end{equation}

If $\SmallSysNode$ is conservative, then $K_s$ defined in \eqref{eq:conskernelexp} factorizes into
\begin{equation}\label{eq:conskernel}
  K_s(\mu,\mu_*,\lambda,\lambda_*)=
  \bbm{H(\mu)\\G(\mu_*)}\bbm{H(\lambda)^*&G(\lambda_*)^*},
  \quad \mu,\lambda,\mu_*,\lambda_*\in\cplus.
\end{equation}
\end{enumerate}
\end{prop}

Clearly, $H(\mu)=e_o(\mu)\,\Cfrak$ and $G(\mu_*)=e_c(\mu_*)\,\Bfrak^*$. The previous result also extends (I.4.73) and the first formula in Prop.\ I.5.8, since the realization of $\varphi$ is arbitrary\,--\,only suitable energy properties are assumed. 

Assuming that $\varphi\in\Sscr(\cplus;\Uscr,\Yscr)$, the kernel $K_s$ has removable singularities at $\mu_*=\overline\lambda$ and $\mu=\overline{\lambda_*}$. When we remove these singularities by continuity, the kernel becomes holomorphic with its values being bounded operators on $\sbm{\Yscr\\\Uscr}$. In the sequel we ignore removable singularities, assuming that they have been removed.

\begin{proof}
Let $(\Uscr,\Xscr,\Yscr)$ denote the input/state/output spaces of $\SmallSysNode$. We begin by proving \eqref{eq:enpreskernel}. By (I.3.6), every system node satisfies
$$
  \SysNode\bbm{G(\lambda_*)^*\\1} = \bbm{\overline{\lambda_*} G(\lambda_*)^*\\\varphi(\overline{\lambda_*})},\quad \lambda_*\in\cplus.
$$
For $\SmallSysNode$ energy preserving, (I.3.14) gives that for all $\lambda_*,\mu_*\in\cplus$, $\nu,v\in\Uscr$:
$$
\begin{aligned}
  &\Ipdp{\overline{\lambda_*}\, G(\lambda_*)^*\nu}{G(\mu_*)^*v}_\Xscr +
  \Ipdp{G(\lambda_*)^*\nu}{\overline{\mu_*}\, G(\mu_*)^*v}_\Xscr \\
  &\qquad =  \Ipdp{\nu}{v}_\Uscr - \Ipdp{\varphi(\overline{\lambda_*})\nu}{\varphi(\overline{\mu_*})v}_\Yscr.
\end{aligned}
$$
This implies that 
$$
  (\overline{\lambda_*}+\mu_*)\,G(\mu_*)\,G(\lambda_*)^*=1-\varphi(\overline{\mu_*})^*\varphi(\overline{\lambda_*}),
$$
i.e., that \eqref{eq:enpreskernel} holds. That \eqref{eq:coenpreskernel} holds for a co-energy preservation system node follows by applying \eqref{eq:enpreskernel} to the energy-preserving system node $\SmallSysNode^*$; recall that the transfer function of this dual system is $\widetilde\varphi$ and that $(\widetilde\varphi)\,\widetilde{}=\varphi$.

A conservative system is by definition both energy-preserving and co-energy preserving, and so \eqref{eq:coenpreskernel} and \eqref{eq:enpreskernel} both hold. Moreover, by \eqref{eq:resdiff} every system node satisfies
$$
  H(\mu) \, G(\lambda_*)^* = \displaystyle\frac{\varphi(\overline{\lambda_*})-\varphi(\mu)}{\mu-\overline{\lambda_*}}$$
and this implies
$$
  G(\mu_*) \, H(\lambda)^* = \frac{\widetilde\varphi(\overline\lambda)-\widetilde\varphi(\mu_*)}{\mu_*-\overline{\lambda}}.
$$

To establish the unitary of $\Cfrak$ from $\Rfrak^\dagger$ onto $\Hscr_o$ in assertion 1, we observe that $K_\Cfrak=K_o$ in the co-energy preserving case. This implies that $\Hscr_\Cfrak=\Hscr_o$ and then unitarity follows from Cor.\ \ref{lem:multspace}. Analogously, $\Bfrak^*$ maps $\Rfrak$ unitarily onto $\Hscr_c$, which implies that $\Bfrak$ is an isometry into $\Xscr$ with range $\Rfrak$.
\end{proof}

Alternatively, \eqref{eq:coenpreskernel} can be inferred from Thm.\ I.5.2 and \eqref{eq:enpreskernel} can also be seen as a consequence of Thm.\ I.4.3. The existence of a conservative realization of an arbitrary operator Schur function on $\cplus$ has been proved in, e.g., \cite{ArNu96,BaSt06}. Formula \eqref{eq:conskernel} provides a Kolmogorov factorization of $K_s$, which proves that $K_s$ is positive, hence the reproducing kernel of a Hilbert space. In order to keep the present article (together with Part I) self-contained, we provide a short direct proof of the positivity in the style of Part I and \cite[pp. 3321--3323]{ArStConserv}. 

\begin{lem}\label{lem:kernelpos}
For $\varphi\in\Sscr(\cplus;\Uscr,\Yscr)$, the function $K_s$ in \eqref{eq:conskernelexp} is a positive kernel on $(\cplus\times\cplus)^2$, i.e., for all $g_k\in\Yscr$, $v_k\in\Uscr$, and $\omega_k,\zeta_k\in\cplus$, $k=1,\ldots,N$, we have
$$
  \sum_{j,k=1}^N\Ipdp{\bbm{g_j\\v_j}}{K_s(\omega_j,\zeta_j,\omega_k,\zeta_k)\bbm{g_k\\v_k}}\geq0.
$$
Let $\Gamma$ be the past/future map determined by $\varphi$ in \eqref{eq:Gamma} and let $e(\mu,\mu_*)=\sbm{e_+(\mu)&0\\0&e_+(\mu_*)}$ be point-evaluation of functions in $\sbm{H^2(\cplus;\Yscr)\\H^2(\cplus;\Uscr)}$. Then the kernel can be factorized as
\begin{equation}\label{eq:kernfact}
  K_s(\mu,\mu_*,\lambda,\lambda_*) =
  e(\mu,\mu_*)\,\Iscr\bbm{1&\Gamma\\\Gamma^*&1}\Iscr^*\,
  e(\lambda,\lambda_*)^*,
\end{equation}
where $\Iscr=\sbm{\iota_o&0\\0&\iota_c}:\sbm{\Hscr_o\\\Hscr_c}\to\sbm{H^2(\cplus;\Yscr)\\H^2(\cplus;\Uscr)}$ is the injection and $\sbm{1&\Gamma\\\Gamma^*&1}$ is positive semidefinite on $\sbm{\Hscr_o\\\Hscr_c}$.
\end{lem}

\begin{proof}
By Thm.\ I.2.4, we have $\Iscr^*=\sbm{1-T_\varphi T_\varphi^*&0\\0&1-T_{\widetilde\varphi} T_{\widetilde\varphi}^*}$, and then Prop.\ \ref{prop:pfmap} and Lem.\ I.2.1.3 give
\begin{equation}\label{eq:kernelfact}
\begin{aligned}
 & K_s(\mu,\mu_*,\lambda,\lambda_*) = \\ &\qquad e(\mu,\mu_*) \,\Iscr \bbm{1&\Gamma\\\Gamma^*&1}\bbm{1-T_\varphi T_\varphi^*&0\\0&1-T_{\widetilde\varphi} T_{\widetilde\varphi}^*} e(\lambda,\lambda_*)^*.
\end{aligned}
\end{equation}
which proves \eqref{eq:kernfact}. Furthermore, $\sbm{1&\Gamma\\\Gamma^*&1}$ is positive semidefinite on $\sbm{\Hscr_o\\\Hscr_c}$ due to the contractivity of $\Gamma$. Now the positivity follows upon observing that 
$$
\begin{aligned}
 \sum_{j,k=1}^N\Ipdp{\bbm{g_j\\v_j}}{K_s(\omega_j,\zeta_j,\omega_k,\zeta_k)\bbm{g_k\\v_k}} &= \\
  \Ipdp{\Iscr^*\sum_{j=1}^Ne(\omega_j,\zeta_j)^*\bbm{g_j\\v_j}}{\bbm{1&\Gamma\\\Gamma^*&1}\Iscr^* \sum_{k=1}^Ne(\omega_k,\zeta_k)^*\bbm{g_k\\v_k}}&\geq0.
\end{aligned}
$$
\end{proof}
 
In order to fit into standard reproducing kernel Hilbert space (RKHS) theory, we can alternatively interpret $K_s(\mu,\mu_*,\lambda,\lambda_*)$ as a kernel function of two variables $\bmu := (\mu, \mu_{*})$ and $\blam := (\lambda, \lambda_{*})$, both in $\cplus \times \cplus$. Then our positive kernel function has the special $2 \times 2$-block form
 $$
   K_s(\bmu, \blam) = \begin{bmatrix} K_{11}(\mu, \lambda) & K_{12}(\mu, 
   \lambda_*) \\ K_{21}(\mu_*, \lambda) & K_{22}(\mu_*, \lambda_*) 
\end{bmatrix}.
 $$
 
 Elements of the corresponding RKHS, which we denote by $\Hscr_s$, are densely spanned by the kernel functions
 $ K_s(\cdot, \blam) \sbm{y \\ u}$, where $\blam$ sweeps $\cplus\times\cplus$ and $\sbm{ y \\ u }$ sweeps $\cY \oplus \cU$.  Note that 
 each such function is a column $\sbm{ f \\ g }(\bmu)$ of the form 
 $$
  \begin{bmatrix} f \\ g \end{bmatrix}(\bmu) = \begin{bmatrix} f(\mu) 
      \\ g(\mu_{*}) \end{bmatrix}
 $$
 where $f$ and $g$ are analytic on $\cplus$; therefore this property continues to hold for 
 all elements of $\Hscr_s$. By standard RKHS theory,  the reproducing property is
 \begin{equation}   \label{genreprod}
 \left\langle \begin{bmatrix} f \\ g \end{bmatrix}, K_s(\cdot, 
 \blam) \begin{bmatrix} \gamma \\ \nu \end{bmatrix} \right\rangle_{\cH(K_s)} = \left\langle 
 \begin{bmatrix} f \\ g \end{bmatrix}(\blam),  \begin{bmatrix} \gamma \\  \nu 
     \end{bmatrix} \right\rangle_{\sbm{\cY\\\cU}}.
 \end{equation}
Taking $\nu = 0$ and setting  (the first column of $K_s(\bmu, \blam)$)
$$  
 K_s^{(1)}(\bmu, \lambda) := \begin{bmatrix} K_{11}(\mu, \lambda) \\ 
   K_{21}(\mu_*, \lambda) \end{bmatrix} = \begin{bmatrix} \displaystyle \frac{1 - 
   \varphi(\mu) \,\varphi(\lambda)^*}{\mu + \overline{\lambda}} \\
  \displaystyle \frac{ \widetilde\varphi(\overline\lambda) - \widetilde\varphi(\mu_*)}{\mu_* - 
   \overline{\lambda}} \end{bmatrix}
$$
then gives
$$
 \left \langle \begin{bmatrix} f \\ g \end{bmatrix}, 
 K^{(1)}_s(\cdot, \lambda) \gamma \right\rangle_{\cH_s} =
 \langle f(\lambda), \gamma \rangle_{\cY}.
$$
Similarly, taking $\gamma=0$ gives
$$
\begin{aligned}
 \left \langle \begin{bmatrix} f \\ g \end{bmatrix}, 
 K^{(2)}_{s}(\cdot, \lambda_*) \nu 
 \right\rangle_{\cH_s} &= \langle g(\lambda_*), \nu 
 \rangle_{\cU},\qquad\text{where} \\
 K_s^{(2)}(\bmu, \lambda_*) := \begin{bmatrix} K_{12}(\mu, \lambda_*) 
 \\ K_{22}(\mu_*, \lambda_*) \end{bmatrix} &=
   \begin{bmatrix} \displaystyle 
   \frac{\varphi(\overline{\lambda_*})-\varphi(\mu)}{\mu - \overline{\lambda_*}} \\
     \displaystyle \frac{1 - \widetilde \varphi(\mu_*) \,\widetilde 
     \varphi(\lambda_*)^*}{\mu_* + \overline{\lambda_*}} 
 \end{bmatrix} \qquad \text{ (column two)},
\end{aligned}
$$
and the general case \eqref{genreprod} is the superposition of these two.

\begin{thm}\label{thm:HsStruct}
Let $W$ denote the positive semidefinite square root of $\sbm{1&\Gamma\\\Gamma^*&1}\in\Lscr\left(\sbm{\Hscr_o\\\Hscr_c}\right)$. Then:
\begin{enumerate}
\item 
We have $\Hscr_s=\range{W}\subset\sbm{\Hscr_o\\\Hscr_c}$ with the lifted norm
\begin{equation}\label{eq:rangechar}
  \|Wh\|_{\Hscr_s}=\|P_{\Ker W^\perp}h\|_{\sbm{\Hscr_o\\\Hscr_c}},\qquad h\in\bbm{\Hscr_o\\\Hscr_c}.
\end{equation}

\item The operators $\sbm{\Gamma\\1}:\Hscr_c\to\Hscr_s$ and $\sbm{1\\\Gamma^*}:\Hscr_o\to\Hscr_s$ are isometric.

\item Setting $\Rfrak_s^\dagger:=\sbm{1\\\Gamma^*}\Hscr_o$ and $\Rfrak_s:=\sbm{\Gamma\\1}\Hscr_c$, both with the norm of $\Hscr_s$, we obtain 
\begin{equation}\label{eq:Embeddings}
  \Rfrak_s^\dagger+\Rfrak_s\subset \Hscr_s\subset\bbm{\Hscr_o\\\Hscr_c}
\end{equation}
with the first embedding dense and the second continuous. The adjoint of the injection $\iota:\Hscr_s\to\sbm{\Hscr_o\\\Hscr_c}$ is 
\begin{equation}\label{eq:iotasadj}
  \iota^*=\bbm{1&\Gamma\\\Gamma^*&1}:\bbm{\Hscr_o\\\Hscr_c}\to\Hscr_s,
\end{equation}
and the reproducing kernel $K_s$ of $\Hscr_s$ has the representation
\begin{equation}\label{eq:KsIota}
  e_s(\lambda,\lambda_*)^*\bbm{\gamma\\\nu} = \bbm{1&\Gamma\\\Gamma^*&1} \bbm{e_o(\lambda)^*\gamma\\e_c(\lambda_*)^*\nu},
    \qquad \lambda,\lambda_*\in\cplus,\, \gamma\in\Yscr,\, \nu\in\Uscr.
\end{equation}  
  
\item The subspaces $\Rfrak_s^\dagger$ and $\Rfrak_s$ are closed, both in $\Hscr_s$ and in $\sbm{\Hscr_o\\\Hscr_c}$.

\item The following maps are co-isometries from $\Hscr_s$ onto $\Hscr_o$ and $\Hscr_c$, respectively:
$$
  \Pi_1:\bbm{f\\g}\mapsto f\qquad\text{and}\qquad \Pi_2:\bbm{f\\g}\mapsto g.
$$
The initial subspace of $\Pi_1$ is $\Rfrak_s^\dagger$ and the initial subspace of $\Pi_2$ is $\Rfrak_s$. The operators $\pi_1:=\Pi_1\big|_{\Rfrak_s^\dagger}:\Rfrak_s^\dagger\to\Hscr_o$ and $\pi_2:=\Pi_2\big|_{\Rfrak_s}:\Rfrak_s\to\Hscr_c$ are unitary. 

\item Denote by $P_{\Rfrak_s}$ and $P_{\Rfrak_s^\dagger}$ the orthogonal projections in $\Hscr_s$ onto $\Rfrak_s$ and $\Rfrak_s^\dagger$, respectively, and let $\Ufrak_s:=(\Rfrak_s^\dagger)^\perp$ and  $\Ufrak_s^\dagger:=(\Rfrak_s)^\perp$. Then
$$
\begin{aligned}
  \pi_1^* = \bbm{1\\\Gamma^*},\qquad \pi_2^*&=\bbm{\Gamma\\1},\qquad 
  \Gamma = \Pi_1\,\pi_2^*,\qquad \Gamma^*=\Pi_2\,\pi_1^*, \\
    P_{\Rfrak_s} &= \bbm{\Gamma\\1}\bbm{0&1}\Big|_{\Hscr_s}, \qquad
  P_{\Ufrak_s^\dagger} = \bbm{1\\0}\bbm{1&-\Gamma}\Big|_{\Hscr_s}, \\
  P_{\Rfrak_s^\dagger} &= \bbm{1\\\Gamma^*}\bbm{1&0}\Big|_{\Hscr_s},\qquad 
  P_{\Ufrak_s}=\bbm{0\\1}\bbm{-\Gamma^*&1}\Big|_{\Hscr_s}.
\end{aligned}
$$
\end{enumerate}
\end{thm}

The spaces $\Ufrak_s$ and $\Rfrak_s$ will turn out to be the unobservable and approximately reachable subspaces of the conservative simple model, respectively. 

\begin{proof}[Proof of Thm.\ \ref{thm:HsStruct}]
Modifying the proof of Thm.\ I.2.4 slightly, we obtain \eqref{eq:rangechar} and \eqref{eq:iotasadj}. Due to \eqref{eq:rangechar}, $\sbm{\Gamma\\1}$ is isometric:
\begin{equation}\label{eq:Pi2isom}
\begin{aligned}
	\left\|\bbm{\Gamma x\\x}\right\|^2_{ \Hscr_s} &= \Ipdp{\bbm{1&\Gamma\\\Gamma^*&1}\bbm{0\\x}}
		{\bbm{1&\Gamma\\\Gamma^*&1}\bbm{0\\x}}_{\Hscr_s} \\
	&= \Ipdp{\bbm{0\\x}}{\bbm{1&\Gamma\\\Gamma^*&1}\bbm{0\\x}}_{\sbm{\Hscr_o\\\Hscr_c}}
	=\|x\|^2_{\Hscr_c},
\end{aligned}
\end{equation}
and the isometricity of $\sbm{1\\\Gamma^*}$ is proved the same way. 

Eq.\  \eqref{eq:KsIota} is only a restatement of \eqref{eq:kernelfact}. Trivially, $W^2\sbm{\Hscr_o\\\Hscr_c}\subset W\sbm{\Hscr_o\\\Hscr_c}\subset \sbm{\Hscr_o\\\Hscr_c}$, and this establishes \eqref{eq:Embeddings}, where the first embedding is dense, because the reproducing kernels of $\Hscr_s$ lie in $\Rfrak_s^\dagger+\Rfrak_s$ by \eqref{eq:KsIota}. Moreover, the norm of $\sbm{1&\Gamma\\\Gamma^*&1}$ as an operator on $\sbm{\Hscr_o\\\Hscr_c}$ is at most $2$: Using that $\Gamma$ is a contraction, Cauchy-Schwarz, and completion of squares, we obtain
$$
	\left\|\bbm{1&\Gamma\\\Gamma^*&1}\bbm{f\\g}\right\|^2 \leq
	2\|f\|^2+4\re\Ipdp{f}{\Gamma g}+2\|\Gamma g\|^2 \leq 4\big(\|f\|^2+\|g\|^2\big).
$$
For all $\sbm{f\\g}=Wh$ with $h\in\sbm{\Hscr_o\\\Hscr_c}\ominus \Ker W$ it then holds that
\begin{equation}\label{eq:contcont}
\begin{aligned}
  \left\|\iota \bbm{f\\g}\right\|_{\sbm{\Hscr_o\\\Hscr_c}}^2 
    &= \|Wh\|_{\sbm{\Hscr_o\\\Hscr_c}}^2 = \Ipdp{\bbm{1&\Gamma\\\Gamma^*&1}h}{h}_{\sbm{\Hscr_o\\\Hscr_c}}
      \leq 2 \left\|\bbm{f\\g}\right\|^2_{\Hscr_s},
\end{aligned}
\end{equation}
i.e., $\iota$ is continuous with norm at most $\sqrt 2$. Claims one to three are proved.

Since $\Gamma$ is bounded with a closed domain, it follows immediately that $\Rfrak_s$ is closed in $\sbm{\Hscr_o\\\Hscr_c}$, and by the isometricity of $\sbm{\Gamma\\1}$, $\Rfrak_s$ is closed also in $\Hscr_s$. Furthermore, it follows from $\Rfrak_s^\dagger+\Rfrak_s\subset \Hscr_s$ that $\Pi_2$ and $\pi_2$ are onto $\Hscr_c$:
\begin{equation}\label{eq:PiOnto}
  \Pi_2\,\Hscr_s\supset
  \pi_2\bbm{\Gamma\\1}\Hscr_c=\Hscr_c.
\end{equation}
An analogous argument shows that $\Rfrak_s^\dagger$ is closed and $\Pi_1\,\Hscr_s=\Hscr_o$.

The operator $\Pi_2\big|_{\Rfrak_s}=\pi_2:\Rfrak_s\to\Hscr_c$ is unitary by \eqref{eq:Pi2isom} and \eqref{eq:PiOnto}. Due to \eqref{eq:KsIota} and the isometricity of $\sbm{\Gamma\\1}$, the space $\Rfrak_s$ is the closed linear span of the kernel functions 
\begin{equation}\label{eq:KsKc}
  \bbm{\Gamma\\1}K_c(\cdot,\lambda_*)\,\nu=e_s(\lambda,\lambda_*)^*\bbm{0\\\nu},\quad\lambda_*\in\cplus,\,\nu\in\Uscr,
\end{equation}
($\lambda\in\cplus$ is insignificant) and this implies that $\Ker{\Pi_2}=\Hscr_s\ominus\Rfrak_s$:
\begin{equation}\label{eq:kernelsdenseR}
\begin{aligned}
  0=&\Ipdp{\bbm{f\\g}}{e_s(\lambda,\lambda_*)^*\bbm{0\\\nu}}_{\Hscr_s}=\Ipdp{g(\lambda_*)}{\nu}_\Uscr,\quad \lambda_*\in\cplus,\, \nu\in\Uscr \\
  &\qquad \iff \quad g=0.
\end{aligned}
\end{equation}
Splitting $\Hscr_s=\sbm{\Rfrak_s \\ \Rfrak_s^\perp}$, we thus obtain $\Pi_2=\bbm{\pi_2&0}$, and furthermore, by the unitarity of $\pi_2$:
$$
	\Pi_2\,\Pi_2^*=\bbm{\pi_2&0}\bbm{\pi_2^*\\0}=1.
$$
Hence, the operator $\Pi_2$ is a co-isometry with initial space $\Rfrak_s$ (and final space $\Hscr_c$).

The unitarity of $\pi_2$ implies that $\pi_2^*=\pi_2^{-1}$ and this operator equals $\sbm{\Gamma\\1}$, because $\pi_2\sbm{\Gamma\\1}x=x$ for all $x\in\Hscr_c$; premultiply by $\pi_2^*$. Then the formula $\Gamma=\Pi_1\pi_2^*$ trivially follows. Moreover, $P_{\Rfrak_s}=\pi_2^*\,\Pi_2$, because $(\pi_2^*\,\Pi_2)^2 = \pi_2^*\,\Pi_2$, $\range{\pi_2^*\,\Pi_2}=\pi_2^* \Hscr_c=\Rfrak_s$, and $\Ker{\pi_2^*\,\Pi_2}=\Ker{\Pi_2}=\Hscr_s\ominus\Rfrak_s$. Then 
$$
	P_{\Ufrak_s^\dagger}=\bbm{1&0\\0&1} -\bbm{0&\Gamma\\0&1}=\bbm{1\\0}\bbm{1&-\Gamma}.
$$

The claims on $\Pi_1$, $\pi_1$, $\Gamma^*$, $P_{\Rfrak_s^\dagger}$, and $P_{\Ufrak_s}$ are proved in the same way.
\end{proof}

We will need the following extension of Prop.\ I.2.6:

\begin{cor}\label{cor:xtendstozero}
Every $\sbm{f\\g}\in\Hscr_s$ satisfies $f(\mu)\to0$ in $\Yscr$ as $\re \mu\to+\infty$ and $g(\mu_*)\to 0$ in $\Uscr$ as $\re\mu_*\to+\infty$. More precisely, for all $\sbm{f\\g}\in\Hscr_s$ and $\mu,\mu_*\in\cplus$:
\begin{equation}\label{eq:H2Zspeed}
  \left\|f(\mu)\right\|_\Yscr \leq \frac{ \sqrt2\left\|\sbm{f\\g}\right\|_{\Hscr_s}}{\sqrt {2\re\mu}}\qquad\text{and}\qquad 
  \left\|g(\mu_*)\right\|_\Uscr \leq \frac{\sqrt2\left\|\sbm{f\\g}\right\|_{\Hscr_s}}{\sqrt {2\re\mu_*}}.
\end{equation}
\end{cor}

\begin{proof}
By Prop.\ I.2.6 and Thm.\ I.2.4.2, for all $f\in\Hscr_o$ and $g\in\Hscr_c$:
$$
 \|f(\mu)\|_\Yscr \leq \frac{\|f\|_{H^2(\cplus;\Yscr)}}{\sqrt {2\re\mu}} \leq \frac{\|f\|_{\Hscr_o}}{\sqrt {2\re\mu}}
    \leq \frac{\left\|\sbm{f\\g}\right\|_{\sbm{\Hscr_o\\\Hscr_c}}}{\sqrt {2\re\mu}},\quad \mu\in\cplus.
$$
Restricting to $\Hscr_s$ and combining this with \eqref{eq:contcont} completes the argument for $f$, and $g$ is handled the same way.
\end{proof}

We end the section with an analogue of Thm.\ \ref{thm:iomaps}. It is needed for our uniqueness result for the conservative simple model, which is a variation on Thm.\ \ref{thm:sharperint}. Inspired by \cite[\S10]{ArKuStCanon}, we define the \emph{(frequency-domain) bilateral input map} of a \emph{passive} system $\SmallSysNode$ with state space $\Xscr$ as the mapping
\begin{equation}\label{eq:BfullDef}
  \Bfrak_{bil}:\bbm{1&\Gamma\\\Gamma^*&1}\bbm{f\\g} \mapsto \bbm{\Cfrak^*&\Bfrak}\bbm{f\\g},\qquad\bbm{f\\g}\in\bbm{\Hscr_o\\\Hscr_c},
\end{equation}
as a first step defined on the dense subspace $\sbm{1&\Gamma\\\Gamma^*&1}\sbm{\Hscr_o\\\Hscr_c}$ of $\Hscr_s$ with range in $\Xscr$. We shall in a moment prove that this mapping can be extended to a contraction $\Hscr_s\to\Xscr$, and its contractive adjoint is called the \emph{(frequency-domain) bilateral output map} 
\begin{equation}\label{eq:CfrakFull}
  \Cfrak_{bil}:=\Bfrak_{bil}^*:\Xscr\to\Hscr_s.
\end{equation}

\begin{thm}\label{thm:bilIO} 
Let $\SmallSysNode$ be a passive system node with transfer function $\varphi$. 
\begin{enumerate}
\item The bilateral input map $\Bfrak_{bil}$ in \eqref{eq:BfullDef} is a contraction $\Hscr_s\to\Xscr$ and $\crange{\Bfrak_{bil}}=\overline{\Rfrak+\Rfrak^\dagger}$, where the closure is in $\Xscr$. Hence, $\SmallSysNode$ is simple if and only if $\Bfrak_{bil}$ has dense range which holds if and only if $\Cfrak_{bil}$ is injective.

\item Denoting the inverse of the injection $\iota:\Hscr_s\to\sbm{\Hscr_o\\\Hscr_c}$ with domain $\iota\,\Hscr$ by $\iota^{-1}$, the bilateral output map can be written more explicitly as
$$
  \Cfrak_{bil}=\iota^{-1}\bbm{\Cfrak\\\Bfrak^*},
$$
i.e, $\Cfrak_{bil}$ is $\sbm{\Cfrak\\\Bfrak^*}$ with range space $\Hscr_s$ rather than $\sbm{\Hscr_o\\\Hscr_c}$. 

\item The bilateral input map has the following action on the kernel functions of $\Hscr_s$: 
\begin{equation}\label{eq:BfullKern}
  \Bfrak_{bil}\, e_s(\lambda,\lambda_*)^*\bbm{\gamma \\ \nu}
    = (\overline\lambda-A^d_{-1})^{-1}C^* \,\gamma +(\overline{\lambda_*}-A_{-1})^{-1} B\, \nu,
\end{equation}
$\lambda,\lambda_*\in\cplus,\, \gamma\in\Yscr,\,\nu\in\Uscr$, and in the notation of Prop.\ \ref{prop:kolmogorov}.3:
\begin{equation}\label{eq:CfullKern}
\begin{aligned}
  \Cfrak_{bil}\,\left( (\overline\lambda-A^d_{-1})^{-1}C^* \,\gamma +(\overline{\lambda_*}-A_{-1})^{-1} B\, \nu\right)& = \\ 
  \iota^{-1}\left((\mu,\mu_*)\mapsto \bbm{H(\mu)\\G(\mu_*)}\bbm{H(\lambda)^*&G(\lambda_*)^*}\bbm{\gamma \\ \nu}
  \right).\end{aligned}
\end{equation}

\item The operator $\Cfrak_{bil}$ maps $\overline{\Rfrak+\Rfrak^\dagger}$ unitarily onto $\Hscr_{\Cfrak_{bil}}$, the RKHS with reproducing kernel
$$
  K_{\Cfrak_{bil}}(\mu,\mu_*,\lambda,\lambda_*):=
  \bbm{H(\mu)\\G(\mu_*)}\bbm{H(\lambda)^*&G(\lambda_*)^*},
  \quad \mu,\mu_*,\lambda,\lambda_*\in\cplus.
$$
We have the alternative characterization
$$
	\Hscr_{\Cfrak_{bil}}=\range{\Cfrak_{bil}}
	\qquad\text{with lifted norm}\qquad
	\|\Cfrak_{bil}\,x\|_{\Hscr_{\Cfrak_{bil}}}=
	\|P_{\overline{\Rfrak+\Rfrak^\dagger}}\,x\|_\Xscr.
$$

\item Assume that $\SmallSysNode$ is conservative. Then $\Hscr_{\Cfrak_{bil}}=\Hscr_s$ and $\Bfrak_{bil}$ maps $\Hscr_s$ isometrically into $\Xscr$, unitarily onto $\overline{\Rfrak+\Rfrak^\dagger}$. Moreover, $\Cfrak^*\Gamma\Bfrak^*\big|_\Rfrak=P_{\Rfrak^\dagger}\big|_\Rfrak$
\end{enumerate}
\end{thm}

\begin{proof}
For every $h\in\sbm{\Hscr_o\\\Hscr_c}$, we have (using \eqref{eq:BfullDef}, $\Gamma=\Cfrak\,\Bfrak$ with $\Cfrak$ and $\Bfrak$ contractive, and \eqref{eq:rangechar}):
$$
\begin{aligned}
  \left\|\Bfrak_{bil}\bbm{1&\Gamma\\\Gamma^*&1}h\right\|_\Xscr^2 
  &= \Ipdp{\bbm{\Cfrak\\\Bfrak^*} \bbm{\Cfrak^*&\Bfrak}h}{h}_{\sbm{\Hscr_o\\\Hscr_c}}
  \leq \Ipdp{\bbm{1&\Gamma\\\Gamma^*&1}h}{h}_{\sbm{\Hscr_o\\\Hscr_c}} \\
 & =  \left\|\bbm{1&\Gamma\\\Gamma^*&1}h\right\|_{\Hscr_s}^2;
\end{aligned}
$$
hence $\Bfrak_{bil}$ is contractive on $\range{\sbm{1&\Gamma\\\Gamma^*&1}}$ which is dense in $\Hscr_s$. By \eqref{eq:iotasadj} and \eqref{eq:BfullDef} it holds that $\Bfrak_{bil}\,\iota^*=\bbm{\Cfrak^*&\Bfrak}:\sbm{\Hscr_o\\\Hscr_c}\to\Xscr$, and this implies that 
$$
  \iota\,\Cfrak_{bil} = (\Bfrak_{bil}\,\iota^*)^* = \bbm{\Cfrak\\\Bfrak^*}.
$$
From here it immediately follows that 
$$
\begin{aligned}
	\ker{\Cfrak_{bil}} &= \ker{\Cfrak}\cap\ker{\Bfrak^*}=\Ufrak\cap\Ufrak^\dagger,\qquad\text{so that} \\ 
	\crange{\Bfrak_{bil}} &= (\Ufrak\cap\Ufrak^\dagger)^\perp=\overline{\Rfrak^\dagger+\Rfrak}.
\end{aligned}
$$

Formula \eqref{eq:BfullKern} is established via \eqref{eq:BfullDef}, \eqref{eq:KsIota}, \eqref{eq:inputaction}, and \eqref{eq:outputaction}:
$$
\begin{aligned}
	  \Bfrak_{bil}\, e_s(\lambda,\lambda_*)^*\bbm{\gamma \\ \nu} 
	&= \bbm{\Cfrak^*&\Bfrak}\bbm{e_o(\lambda)^*\,\gamma \\ e_c(\lambda_*)^*\,\nu} \\
    &= (\overline\lambda-A^d_{-1})^{-1}C^* \,\gamma +(\overline{\lambda_*}-A_{-1})^{-1} B\, \nu;
\end{aligned}
$$
compare this to Def.\ \ref{def:simple} to obtain the characterizations of simplicity. Finally, \eqref{eq:dualio} and \eqref{eq:HGdef} give \eqref{eq:CfullKern}; the factor $\iota^{-1}$ emphasizes the fact that $\Cfrak_{bil}$ maps into $\Hscr_s$ rather than $\sbm{\Hscr_o\\\Hscr_c}$. This completes the proof of assertions one to three.

Assertion four follows from Lem.\ \ref{lem:multspace} upon observing that \eqref{eq:outputaction}, \eqref{eq:dualio}, and \eqref{eq:HGdef} imply that 
$$
	\Cfrak_{bil}\,x=(\mu,\mu_*)\mapsto\bbm{H(\mu)\,x\\G(\mu_*)\,x},
	\qquad x\in\Xscr,
$$
and that by the above, $\Ker{\Cfrak_{bil}}^\perp=\overline{\Rfrak+\Rfrak^\dagger}$. For the rest of the proof, we assume that $\SmallSysNode$ is conservative. Then Prop.\ \ref{prop:kolmogorov}.3 gives that $K_{\Cfrak_{bil}}=K_s$ and by assertion 4, $\Cfrak_{bil}$ maps $\overline{\Rfrak+\Rfrak^\dagger}$ (which is isometrically contained in $\Xscr$) unitarily onto $\Hscr_s$; then $\Bfrak_{bil}=\Cfrak_{bil}^*$ maps $\Hscr_s$ isometrically into $\Xscr$, with $\range{\Bfrak_{bil}}=\overline{\Rfrak+\Rfrak^\dagger}$ by the above; now $\Bfrak_{bil}$ has closed range because it is isometric with a closed domain. By Prop.\ \ref{prop:kolmogorov}, $\Cfrak^*$ and $\Bfrak$ are both isometric into $\Xscr$; hence $\Cfrak^*\Cfrak=P_{\Rfrak^\dagger}$ and $\Bfrak\Bfrak^*=P_{\Rfrak}$. Combining this with the Def.\ \ref{def:pfmap} of $\Gamma$ gives $\Cfrak^*\Gamma\Bfrak^*\big|_\Rfrak=P_{\Rfrak^\dagger}\big|_\Rfrak$.
\end{proof}

We remark that $\Cfrak^*\Gamma\Bfrak^*\big|_\Rfrak=P_{\Rfrak^\dagger}\big|_\Rfrak$ for a conservative system node means that  $\Gamma=P_{\Rfrak^\dagger}\big|_\Rfrak$ if we make the unitary identification of $y\in\Hscr_o$ with $\Cfrak^* y$ and, similarly, we identify $x\in\Rfrak$ with $\Bfrak^*x$.

\section{The conservative simple functional model}\label{sec:consmodel}

We construct the conservative simple realization in the following way; cf.\ Lem.\ I.4.1:

\begin{prop}\label{prop:consreal}
Let $\varphi\in\Sscr(\cplus;\Uscr,\Yscr)$ and let $\Hscr_s$ be the Hilbert space with reproducing kernel \eqref{eq:conskernelexp}. The mapping
\begin{equation}\label{eq:SsDefprel}
\begin{aligned}
  \SysNode_s:&\bbm{e_s(\lambda,{\lambda_*})^*\\\bbm{\widetilde\varphi(\overline\lambda)&1}}\bbm{\gamma\\\nu}\mapsto
  \bbm{e_s(\lambda,{\lambda_*})^*\bbm{-\overline\lambda&0\\0&\overline{\lambda_*}}\\\bbm{1&\varphi(\overline{\lambda_*})}} \bbm{\gamma\\\nu}, \\
  &\qquad \lambda,\lambda_*\in\cplus,\,\gamma\in\Yscr,\,\nu\in\Uscr,
\end{aligned}
\end{equation}
extends via linearity and operator closure to define a scattering-isometric closed linear operator $\SmallSysNode_s:\sbm{\Xscr\\\Uscr}\supset\dom{\SmallSysNode_s}\to\sbm{\Xscr\\\Yscr}$.
\end{prop}

\begin{proof}
From $e_s(\mu,{\mu_*})e_s(\lambda,{\lambda_*})^*=K_s(\mu,{\mu_*},\lambda,{\lambda_*})$ and \eqref{eq:conskernelexp}, we obtain the following:
$$
\begin{aligned}
  &\bbm{-\mu&0\\0&{\mu_*}}e_s(\mu,\mu_*)\,e_s(\lambda,{\lambda_*})^*+e_s(\mu,{\mu_*})\,e_s(\lambda,{\lambda_*})^*\bbm{-\overline\lambda&0\\0&\overline{\lambda_*}} \\
  &\quad = \bbm{-\mu&0\\0&{\mu_*}}K_s(\mu,{\mu_*},\lambda,{\lambda_*})+K_s(\mu,{\mu_*},\lambda,{\lambda_*})\bbm{-\overline\lambda&0\\0&\overline{\lambda_*}} \\
  &\quad = \bbm{\varphi(\mu)\varphi(\lambda)^*-1 & \varphi(\mu)-\varphi(\overline{\lambda_*}) \\ 
    \widetilde\varphi(\overline\lambda)-\widetilde\varphi({\mu_*}) & 1-\widetilde\varphi({\mu_*})\widetilde\varphi({\lambda_*})^*} \\
  &\quad = \bbm{1&\widetilde\varphi(\overline\mu)^*\\\varphi(\overline{\mu_*})^*&1} \bbm{-1&0\\0&1} \bbm{1&\varphi(\overline{\lambda_*})\\\widetilde\varphi(\overline\lambda)&1},
\end{aligned}
$$
i.e., for all $\nu,\eta\in\Uscr$, $\gamma,\xi\in\Yscr$, and $\mu,\mu_*,\lambda,\lambda_*\in\cplus$:
\begin{equation}\label{eq:consenpres}
\begin{aligned}
  &\Ipdp{e_s(\lambda,{\lambda_*})^*\bbm{\gamma\\\nu}}{e_s(\mu,\mu_*)^*
  	\bbm{-\overline\mu&0\\0&\overline{\mu_*}}\bbm{\xi\\\eta}}_{\Hscr_s} \\
  &\qquad \qquad + \Ipdp{e_s(\lambda,{\lambda_*})^*\bbm{-\overline\lambda&0\\0&\overline{\lambda_*}}\bbm{\gamma\\\nu}}{e_s(\mu,\mu_*)^*\bbm{\xi\\\eta}}_{\Hscr_s} \\
  &\quad =\Ipdp{\bbm{-1&0\\0&1}\bbm{1&\varphi(\overline{\lambda_*}) \\
    \widetilde\varphi(\overline\lambda)&1}\bbm{\gamma\\\nu}}{\bbm{1&\varphi(\overline{\mu_*})\\\widetilde\varphi(\overline\mu)&1} \bbm{\xi\\\eta}}_{\sbm{\Yscr\\\Uscr}}.
  \end{aligned}
\end{equation}

Proceeding along the lines of the proof of Lemma I.4.1, one obtains from \eqref{eq:consenpres} that the extension by linearity and operator closure of the mapping \eqref{eq:SsDefprel} is a scattering-isometric, well-defined single-valued operator; please note that $\sbm{z_1\\z_2}\perp x_2$ for all $\sbm{x_2\\u_2}\in\dom{\SmallSysNode_s}$ implies
$$
  0=\Ipdp{\bbm{z_1\\z_2}}{e_s(\lambda,\lambda_*)^*\bbm{\xi\\\eta}}_{\Hscr_s}=\Ipdp{z_1(\lambda)}{\xi}_\Yscr+\Ipdp{z_2(\lambda_*)}{\eta}_\Uscr
$$
for all $\lambda,\lambda_*\in\cplus$, $\xi\in\Yscr$, and $\eta\in\Uscr$, i.e., $z_1=0$ and $z_2=0$.
\end{proof}

From now on $\SmallSysNode_s$ always denotes the extension of the mapping \eqref{eq:SsDefprel} by linearity and operator closure.

\begin{thm}\label{thm:basic}
For all $\varphi\in\Sscr(\cplus;\Uscr,\Yscr)$, the following claims are true:
\begin{enumerate}
\item The operator $\SmallSysNode_s$ is a simple scattering-conservative system node with input/state/output spaces $(\Uscr,\Xscr,\Yscr)$ and transfer function $\varphi$. 

\item The adjoint $\SmallSysNode_s^*:\sbm{\Xscr\\\Yscr}\supset\dom{\SmallSysNode_s^*}\to\sbm{\Xscr\\\Uscr}$ is the extension by linearity and operator closure of
\begin{equation}\label{eq:SsAdjSpec}
\begin{aligned}
  \SysNode_s^*:\bbm{e_s(\lambda,{\lambda_*})^*\\ \bbm{1&\varphi(\overline{\lambda_*})}}\bbm{\gamma\\\nu}\mapsto
  \bbm{e_s(\lambda,{\lambda_*})^*\bbm{\overline\lambda&0\\0&-\overline{\lambda_*}}\\ \bbm{\widetilde\varphi(\overline\lambda)&1}} \bbm{\gamma\\\nu}, \\
  \quad \lambda,\lambda_*\in\cplus,\,\gamma\in\Yscr,\,\nu\in\Uscr.
\end{aligned}
\end{equation}

\item The (unilateral) input map of $\SmallSysNode_s$ is $\Bfrak_s=\sbm{\Gamma\\1}$ and the approximately reachable subspace is $\Rfrak_s=\range{\Bfrak_s}=\sbm{\Gamma\\1}\Hscr_c$. The adjoint of the (unilateral) output map of $\SmallSysNode_s$ is $\Cfrak_s^*=\sbm{1\\\Gamma^*}$ and the orthogonal complement of the unobservable subspace is $\Rfrak_s^\dagger=\range{\Cfrak_s^*}=\sbm{1\\\Gamma^*}\Hscr_o$. The bilateral input and output maps of $\SmallSysNode_s$ are both equal to the identity operator on $\Hscr_s$.
\end{enumerate}
\end{thm}

Assertion three can be written more explicitly as (for $\alpha\in\cplus, \, u \in \cU,\, \sbm{x_1\\x_2}\in\Hscr_s$)
\begin{align}
  (\alpha - A_{s,-1})^{-1} B_s \,u & = (\mu,\mu_*)\mapsto  
  \bbm{\displaystyle \frac{\varphi(\mu)-\varphi(\alpha)}{\alpha-\mu} \\ 
    \displaystyle \frac{1-\widetilde\varphi(\mu_*)\,\varphi(\alpha)}{\alpha+\mu_*}}u,\quad \mu,\mu_* \in {\mathbb C}^+, \label{resB'} \\
  C_s(\alpha - A_s)^{-1}\sbm{x_1\\x_2} &= x_1(\alpha). \label{Cres'}
\end{align}

\begin{proof}[Proof of Thm.\ \ref{thm:basic}]
We obtain that $\SmallSysNode_s$ is an energy-preserving system node by generalizing the proof of Thm.\ I.4.2: If 
$$
	\bbm{x\\u}\in\bbm{\Hscr_s\\\Uscr}\ominus
	\range{\bbm{\bbm{1&0}-\bbm{A_s \& B_s}\\\bbm{0&\sqrt2}}}
$$
then in particular for all $\lambda,\lambda_*\in\cplus$, $\gamma\in\Yscr$, and $\nu\in\Uscr$:
\begin{equation}\label{eq:maxenpres}
\begin{aligned}
	0 &=\Ipdp{\bbm{\bbm{1&0}-\bbm{A_s \& B_s}\\\bbm{0&\sqrt2}}\bbm{e_s(\lambda,{\lambda_*})^*\\\bbm{\widetilde\varphi({\overline\lambda})&1}}\bbm{\gamma\\\nu}}{\bbm{x\\u}}\\ 
	  &=\Ipdp{\bbm{\gamma\\\nu}}{\bbm{1+\overline\lambda&0\\0&1-\overline\lambda_*}\bbm{x_1(\lambda)\\x_2({\lambda_*})}} +\Ipdp{\sqrt 2\,\widetilde\varphi(\overline\lambda)\gamma+\sqrt 2\,\nu}{u}.
\end{aligned}
\end{equation}
Restricting \eqref{eq:maxenpres} to the case $\lambda_*=1$ and $\gamma=0$, we obtain that $u=0$. Keeping $\gamma=0$ but taking $\lambda_*\neq1$, we get $x_2=0$. Finally, letting $\gamma$ run over $\Yscr$ and $\lambda$ over $\cplus$, we obtain that $x_1=0$. Combining this with Prop.\ \ref{prop:consreal} and the proof of Thm.\ I.4.2, we obtain that $\SmallSysNode_s$ is an energy-preserving system node. In the same way we see that the range of $\sbm{\sbm{1&0}+\sbm{A_s \& B_s}\\\sbm{C_s\&D_s}}$ is dense in $\sbm{\Hscr_s\\\Uscr}$; then $\SmallSysNode_s$ is a conservative system node by Thm.\ I.3.12. Claim two now follows immediately from \eqref{eq:tfi}.

Using Def.\ I.3.1 and \eqref{eq:SsDefprel}, we calculate
\begin{equation}\label{eq:resB}
\begin{aligned}
  A_{s,-1}\, e_s(\lambda,{\lambda_*})^*\bbm{0\\\nu}+B_s\nu &= \bbm{A_s \& B_s} \bbm{e_s(\lambda,{\lambda_*})^*\\\bbm{\widetilde\varphi(\overline\lambda)&1}}\bbm{0\\\nu}
    = \overline{\lambda_*}\, e_s(\lambda,{\lambda_*})^*\bbm{0\\\nu} \\
  &\Longrightarrow\quad e_s(\lambda,{\lambda_*})^*\bbm{0\\\nu} = (\overline{\lambda_*}-A_{s,-1} )^{-1} B_s\nu.
\end{aligned}
\end{equation}
From (I.3.5) and \eqref{eq:SsDefprel} we then have (for $\lambda_*\in\cplus,\,\nu\in\Uscr$):
$$
\begin{aligned}
  \widehat\Dfrak(\overline{\lambda_*})\,\nu &=
   \bbm{C_s\&D_s}\bbm{(\overline{\lambda_*}-A_{s,-1} )^{-1} B_s\\1}\nu \\
   &= \bbm{C_s\&D_s}\bbm{e_s(\lambda,{\lambda_*})^*\\\bbm{\widetilde\varphi(\overline\lambda)&1}}\bbm{0\\\nu} 
  = \varphi(\overline{\lambda_*})\,\nu.
\end{aligned}
$$

From \eqref{eq:resB} and \eqref{eq:KsKc} we have that $\Bfrak_s=\sbm{\Gamma\\1}$ and by definition $\Rfrak_s=\crange{\Bfrak_s}$. However, since $\Bfrak_s:\Hscr_c\to\Hscr_s$ is isometric by Thm.\ \ref{thm:HsStruct}.2, $\range{\Bfrak_s}$ is closed. By \eqref{eq:outputaction} we have $e_o(\lambda)\,\Cfrak=C\,(\lambda-A)^{-1}$ and carrying out the calculation \eqref{eq:resB} for $\SmallSysNode_s^*$ in \eqref{eq:SsAdjSpec}, we obtain $\Cfrak_s^*=\sbm{1\\\Gamma^*}$: 
\begin{equation}\label{CfrakSAdj}
	\Cfrak_s^*\,e_o(\lambda)^*\,\gamma = 
	(\overline\lambda-A^d_{s,-1} )^{-1} C_s^*\gamma = 
	e_s(\lambda,{\lambda_*})^*\bbm{\gamma\\0} =
	\bbm{1\\\Gamma^*}e_o(\lambda)^*\,\gamma;
\end{equation}
where we also used \eqref{eq:KsIota}. Combining this and $\Bfrak_s=\sbm{\Gamma\\1}$ with \eqref{eq:BfullDef}, we obtain for all $\sbm{f\\g}\in\sbm{\Hscr_o\\\Hscr_c}$ that
$$
	 \Bfrak_{s,full}\bbm{1&\Gamma\\\Gamma^*&1}\bbm{f\\g} = \bbm{\Cfrak^*&\Bfrak}\bbm{f\\g}
	 	=  \bbm{1&\Gamma\\\Gamma^*&1}\bbm{f\\g};
$$
this shows that $\Bfrak_{s,full}$ acts as the identity on the dense subspace $\range{\sbm{1&\Gamma\\\Gamma^*&1}}$ of $\Hscr_s$; by Thm.\ \ref{thm:bilIO}.4, $\SmallSysNode_c$ is simple. Using \eqref{eq:CfrakFull} we obtain that also $\Cfrak_{s,full}=1$.
\end{proof}

We have the following variant of Thm.\ \ref{thm:sharperint} regarding intertwinement with $\SmallSysNode_s$:

\begin{thm}\label{thm:Intertw} 
Let $\SmallSysNode$ be a conservative system with state space $\Xscr$ and transfer function $\varphi$. Then the bilateral input map $\Bfrak_{bil}$ of $\SmallSysNode$ in Thm.\ \ref{thm:bilIO} intertwines $\SmallSysNode_s$ isometrically with $\SmallSysNode$. Additionally, $\SmallSysNode_s$ is simple if and only if $\Bfrak_{bil}$ is unitary.
\end{thm}

\begin{proof}
The isometricity of $\Bfrak_{bil}$ and the fact that $\range{\Bfrak_{bil}}$ is dense if and only of $\SmallSysNode$ is simple were shown in Thm.\ \ref{thm:bilIO}.

Using \eqref{eq:BfullKern}, Lemma \ref{lem:consdual}, (I.3.6), and \eqref{eq:SsDefprel}, for all $\lambda,\lambda_*\in\cplus$, $\gamma\in\Yscr$, and $\nu\in\Uscr$:
\begin{align*}
  & \bbm{\Bfrak_{bil}\,e_s(\lambda,{\lambda_*})^*\sbm{\gamma\\\nu} \\ \varphi(\overline\lambda)^*\gamma+\nu}  \\
  &\qquad = \bbm{(\overline{\lambda}-A^d_{-1})^{-1}C^*\gamma\\\varphi(\overline\lambda)^*\gamma} + \bbm{(\overline{\lambda_*}-A_{-1})^{-1}B\nu\\\nu}
	\in\dom{\SmallSysNode}\quad\text{and}
\end{align*}
\begin{equation}\label{eq:intertwcalc}
\begin{aligned}
 & \SysNode\bbm{\Bfrak_{bil}\,e_s(\lambda,{\lambda_*})^*\sbm{\gamma\\\nu} \\ \varphi(\overline\lambda)^*\gamma+\nu} \\
  &\qquad = \bbm{-\overline\lambda(\overline\lambda-A^d_{-1})^{-1}C^*\gamma\\\gamma} 
    + \bbm{\overline{\lambda_*}(\overline{\lambda_*}-A_{-1})^{-1}B\nu\\\varphi(\overline{\lambda_*})\,\nu} \\
  &\qquad = \bbm{\Bfrak_{bil}\bbm{A_s\&B_s}\\\bbm{C_s\&D_s}} 
  \bbm{e_s(\lambda,{\lambda_*})^*\sbm{\gamma\\\nu} \\ \varphi(\overline\lambda)^*\,\gamma+\nu}.
\end{aligned}
\end{equation}
Taking linear combinations of elements $\sbm{\overline{\lambda_*}(\overline{\lambda_*}-A_{-1})^{-1}B\nu\\\varphi(\overline{\lambda_*})\,\nu}$ and closing in $\dom{\SmallSysNode_s}$ (equipped with the graph norm), we obtain both \eqref{eq:intertwdom} and \eqref{eq:intertw}.
\end{proof}

As an immediate consequence of the theorem, any two simple conservative realizations with the same transfer function are unitarily similar.

The above formulas \eqref{eq:SsDefprel} and \eqref{eq:SsAdjSpec} for $\SmallSysNode_s$ and its adjoint only give the action on special, kernel-like elements. Using \eqref{eq:tfi}, we can obtain explicit formulas for the action of $\SmallSysNode_s$ on generic elements of its domain: 

\begin{thm}\label{thm:consexplicit}
The model $\SmallSysNode_s$ has the explicit representation given in Thm. \ref{thm:preliminary}, where the vector $y$ can alternatively be defined as the unique $y \in \cY$ for which the function $\sbm{z_1 \\ z_2}$ in \eqref{defz'} is an element of $\cH_s$.
\end{thm}

\begin{proof}
By \eqref{eq:SsAdjSpec}, for all $\sbm{x\\u}\in\sbm{\Hscr_s\\\Uscr}$, $\mu,\mu_*\in\cplus$, $\gamma\in\Yscr$, and $\nu\in\Uscr$:
\begin{equation}\label{eq:consactioncalc}
\begin{aligned}
	\Ipdp{\bbm{x\\u}}{\SysNode_s^*\bbm{e_s(\mu,\mu_*)^*\\\bbm{1&\varphi(\overline{\mu_*})}}
	\bbm{\gamma\\\nu}}_{\sbm{\Hscr_s\\\Uscr}} &= \\
	\Ipdp{\bbm{\mu\,x_1(\mu)\\-\mu_*\,x_2(\mu_*)}
		+\bbm{\varphi(\mu)\\1}u}{\bbm{\gamma\\\nu}}_{\sbm{\Yscr\\\Uscr}}.
\end{aligned}
\end{equation}
First assume that $\sbm{x\\u}\in\dom{\SmallSysNode_s}$ and define $\sbm{z\\y}:=\SmallSysNode_s\sbm{x\\u}\in\sbm{\Xscr\\\Yscr}$, so that the left-hand side of \eqref{eq:consactioncalc} equals
$$
\begin{aligned}
	\Ipdp{\bbm{z\\y}}{\bbm{e_s(\mu,\mu_*)^*\\\bbm{1&\varphi(\overline{\mu_*})}}
	\bbm{\gamma\\\nu}}_{\sbm{\Hscr_s\\\Uscr}} &= 
	\Ipdp{e_s(\mu,\mu_*)z+\bbm{1\\\varphi(\overline{\mu_*})^*}y}{\bbm{\gamma\\\nu}}_{\sbm{\Yscr\\\Uscr}}
\end{aligned}
$$
for all $\sbm{\gamma\\\nu}$. Then \eqref{defz'} holds, and moreover \eqref{defy'} holds by Cor.\ \ref{cor:xtendstozero}. Thus, the action of $\SmallSysNode_s$ is correct and $\dom{\SmallSysNode_s}$ is contained in the set on the right-hand side of \eqref{eq:domdef}.

Now drop the assumption $\sbm{x\\u}\in\dom{\SmallSysNode_s}$ and instead assume that $y\in\Yscr$ is such that $\sbm{z_1\\z_2}$ defined by \eqref{defz'} is in $\Hscr_s$. Then $y$ satisfies \eqref{defy'} and by the definition of $\sbm{z_1\\z_2}$, \eqref{eq:consactioncalc} equals
$$
	\Ipdp{\bbm{z\\y}}{\bbm{e_s(\mu,\mu_*)^*\\\bbm{1&\varphi(\overline{\mu_*})}}
	\bbm{\gamma\\\nu}}_{\sbm{\Hscr_s\\\Uscr}}
$$
for all $\sbm{e_s(\mu,\mu_*)^*\\\sbm{1&\varphi(\overline{\mu_*})}} 	\sbm{\gamma\\\nu}$ which by the definition of $\SmallSysNode_s$ span a dense subspace of $\dom{\SmallSysNode_s}$. Thus 
$$
	\Ipdp{\bbm{x\\u}}{\SysNode_s^*\bbm{x'\\y'}}_{\sbm{\Hscr_s\\\Uscr}} =
	\Ipdp{\bbm{z\\y}}{\bbm{x'\\y'}}_{\sbm{\Hscr_s\\\Uscr}}
$$
for all $\sbm{x'\\y'}\in\dom{\SmallSysNode_s^*}$, so that $\sbm{x\\u}\in\dom{\big(\SmallSysNode_s^*\big)^*}$.
\end{proof}

By (I.3.7), a system node $\SmallSysNode$ with state space $\Xscr$ can be reconstructed from its component operators $A$, $B$, $C$, and its transfer function $\widehat\Dfrak(\alpha)$, for an arbitrary $\alpha\in\res A$. In the following result, we describe these operators for $\SmallSysNode_s$. In order to state the result, we define a linear operator $R_\alpha$, $\alpha\in\cplus$, on the space of analytic functions ${\mathbb C}^+ \to \sbm{\cY\\\cU}$ by
\begin{equation}\label{eq:resanalyt}
	R_\alpha \bbm{x_1\\x_2} :=  (\mu,\mu_*) \mapsto 
	\bbm{ \displaystyle \frac{x_1(\mu) - x_1(\alpha)}{\alpha - \mu} \\ 
		\displaystyle \frac{x_2(\mu_*) - \widetilde\varphi(\mu_*) \, x_1(\alpha)}{\alpha + \mu_*}},
	 \quad \mu,\mu_*\in\cplus.
\end{equation}

\begin{prop}   \label{P:resolA}
The main operator $A_s$ of $\SmallSysNode_s$ is
\begin{equation}\label{eq:As}
  A_s \bbm{x_1\\x_2} =(\mu,\mu_*)\mapsto 
	\bbm{\mu\, x_1(\mu)\\-\mu_*\, x_2(\mu_*)}-\bbm{1\\\widetilde\varphi(\mu_*)} y,\quad\mu,\mu_*\in\cplus,
\end{equation}
with domain consisting of those $\sbm{x_1\\x_2}\in\Hscr_s$ for which the limit 
$$
	y:=\lim_{\re\eta\to+\infty} \eta \,x_1(\eta)
$$ 
exists in $\Yscr$ and the function in \eqref{eq:As} lies in $\Hscr_s$.

The observation operator is 
\begin{equation}\label{eq:Cs}
	C_s\bbm{x_1\\x_2}=\lim_{\re\eta\to+\infty} \eta \,x_1(\eta),\quad \bbm{x_1\\x_2}\in\dom{A_s}.
\end{equation}
For all $\sbm{x_1\\x_2}\in\dom{A_s}$ it holds (with $y$ as above) that
\begin{equation}\label{eq:ylimto0}
	\lim_{\re\eta\to+\infty} \eta\,x_2(\eta)+\widetilde\varphi(\eta)\,y=0.
\end{equation}

The resolvent of $A_s$ is
\begin{equation}   \label{res}
 ( \alpha - A_s)^{-1} = R_\alpha\big|_{\cH_s},
\end{equation}
and moreover, for all $\alpha\in\cplus$ and $\sbm{x_1\\x_2}\in\Hscr_s$:
\begin{equation}\label{eq:Ares'}
  A_s( \alpha - A_s)^{-1} \bbm{x_1\\x_2} = (\mu,\mu_*)\mapsto 
    \bbm{\displaystyle \frac{ \mu\,x_1(\mu) - \alpha\,x_1(\alpha)}{\alpha - \mu} \\ 
    \displaystyle -\frac{ \mu_*\,x_2(\mu_*) + \widetilde\varphi(\mu_*)\,\alpha\,x_1(\alpha)}{\alpha + \mu_*}}, 
\end{equation}
$\mu,\mu_* \in {\mathbb C}^+$.
\end{prop}

\begin{proof}
The claims on $A_s$ and $C_s$ (including \eqref{eq:ylimto0}) follow by comparing Defs.\ I.3.1\,--\,2 to \eqref{defy'}\,--\,\eqref{eq:domdef}. For $\alpha\in\cplus$ and $\sbm{w_1\\w_2} \in \dom{A_s}$ arbitrary, set  $\sbm{x_1\\x_2} := (\alpha - A_s) \sbm{w_1\\w_2}$.  From \eqref{eq:As}--\eqref{eq:Cs} we see that
\begin{equation}\label{eq:alphaminAs}
  \bbm{x_1(\mu)\\x_2(\mu_*)} = \bbm{(\alpha-\mu)\, w_1(\mu)\\(\alpha+\mu_*)\, w_2(\mu_*)}+\bbm{1\\\widetilde\varphi(\mu_*)}C_s\bbm{w_1\\w_2}.
\end{equation}
We conclude that $C_s \sbm{w_1\\w_2} = x_1(\alpha)$, which gives an alternative proof of \eqref{Cres'}, and solving for $\sbm{w_1\\w_2}$, we get \eqref{res}. Then \eqref{eq:Ares'} follows from \eqref{res} and the identity $A_s (\alpha - A_s)^{-1} = \alpha (\alpha - A_s)^{-1} - 1$.
\end{proof}

The following are easy to see (for all $\alpha\in\cplus$):
\begin{align}  
  ( \alpha - A_{s,-1})^{-1} A_{s,-1} \bbm{x_1\\x_2} &= A_s( \alpha - A_s)^{-1} \bbm{x_1\\x_2} \notag \\ &= R_\alpha
    \left((\mu,\mu_*)\mapsto \bbm{ \mu\,x_1(\mu) \\ -\mu_*\,x_2(\mu_*)}\right) \label{eq:Ares} \\
  (\alpha - A_{s,-1})^{-1} B_s u & = R_\alpha
    \left((\mu,\mu_*)\mapsto \bbm{\varphi (\mu) \\ 1}u\right).  \label{resB}
\end{align}
From the former of these formulas, it seems reasonable that $A_{s,-1}\,\sbm{x_1\\x_2}=(\mu,\mu_*)\mapsto \sbm{ \mu\,x_1(\mu) \\ -\mu_*\,x_2(\mu_*)}$ while the latter hints that the control operator $B_s$ of $\SmallSysNode_s$ could be  $B_s\,u=\sbm{\varphi(\cdot)\\1}u$. In order to properly decouple $\bbm{A_s\&B_s}$ into $\bbm{A_{s,-1}&B_s}$ and prove these conjectures, we shall next interpret $\bbm{A_s\&B_s}$ as an operator that maps into the extrapolation space $\Hscr_{s,-1}$ of $\Hscr_s$.

\section{The extrapolation space and its reproducing kernel}\label{sec:extrapolation}

The formula \eqref{res} for the resolvent of $A_s$ suggests a way to concretely identify the $(-1)$-scaled rigged space $\cH_{s,-1}$ defined abstractly as the completion of the space $\cH_s$ in the norm
$$
  \| x\| = \| (\beta - A_s)^{-1}x\|_{\cH_s},
$$
where $\beta$ is the fixed rigging parameter. Indeed, we should have $\sbm{z_1\\z_2}\in\Hscr_{s,-1}$ if and only if $R_\beta\sbm{z_1\\z_2}\in\Hscr_s$, see \eqref{eq:resanalyt}, and in this case
\begin{equation}\label{eq:Hs-1norm}
  	\left\| \bbm{z_1\\z_2} \right\|_{\cH_{s,-1}} = \left\| R_\beta \bbm{z_1\\z_2} \right\|_{\cH_s}.
\end{equation}

It is straightforward to verify that $R_\beta\sbm{z_1\\z_2}=0 \iff \sbm{z_1\\z_2}\in \Zscr_s$, where
\begin{equation}\label{eq:Zscrsdef}
  \Zscr_s:=\set{(\mu,\mu_*)\mapsto \bbm{1\\\widetilde\varphi(\mu_*)}\gamma,\quad \mu,\mu_*\in\cplus,\, \gamma\in\Yscr};
\end{equation}
in particular, $R_\beta\Zscr_s=\zero\subset\Hscr_s$. Hence, $\|\cdot\|_{\Hscr_{s,-1}}$ is a norm on the quotient space
\begin{equation}   \label{eq:Hs-1}
	\cH_{s,-1} := \left\{ \bbm{x_1\\x_2} : \cplus\times\cplus \to \bbm{\cY\\\cU} ~\text{analytic}
		\bigmid R_\beta\bbm{x_1\\x_2} \in \cH_s \right\}\big/\Zscr_s.
\end{equation}
The norm on $\cH_{s,-1}$, and the corresponding inner product, depend on the choice of $\beta \in\cplus$, but different choices of $\beta$ give equivalent norms.

\begin{thm}\label{T:Hs-1concrete}
The space $\Hscr_{s,-1}$ in \eqref{eq:Hs-1} is a Hilbert space with the norm \eqref{eq:Hs-1norm}.
\begin{enumerate}
\item The map $\iota : \sbm{x_1\\x_2} \mapsto \sbm{x_1\\x_2}+\Zscr_s$ embeds $\cH_s$ continuously into $\cH_{s,-1}$ as a dense subspace.

A given element $\sbm{z_1\\z_2}+\Zscr_s \in \cH_{s,-1}$
is of the form $\iota\sbm{x_1\\x_2}$ for some $\sbm{x_1\\x_2} \in \cH_s$ if and only if the function $R_\alpha\sbm{z_1\\z_2}$ is not only in $\cH_s$ but in fact is in $\dom{A_s}\subset \cH_s$, for some, or equivalently for all, $\alpha\in\cplus$.  

When $\sbm{z_1\\z_2}+\Zscr_s =\iota\sbm{x_1\\x_2}$, the representative $\sbm{x_1\\x_2}\in \Hscr_s$ for the equivalence class $\sbm{z_1\\z_2}+\Zscr_s$ is uniquely determined by the decay of the first component at infinity, i.e., by the condition $\lim_{\re \eta \to+ \infty}  x_1(\eta) = 0$.

\item When $\cH_s$ is identified as a linear sub-manifold of $\cH_{s,-1}$ as in assertion one, for all $\sbm{x_1\\x_2}\in\Hscr_s$,
\begin{equation}\label{eq:AsExt}
  A_{s,-1}\left(\bbm{x_1\\x_2}+\Zscr_s\right) = (\mu,\mu_*) \mapsto  \bbm{\mu \,x_1(\mu) \\ -\mu_*\, x_2(\mu_*)}+\Zscr_s,
	\qquad \mu,\mu_*\in\cplus,
\end{equation}
is the unique extension of $A_s$ to a closed operator on $\Hscr_{s,-1}$. Moreover, 
\begin{equation}\label{eq:AsExtRes}
	(\alpha - A_{s,-1})^{-1} = R_\alpha\big|_{\Hscr_{s,-1}},\qquad \alpha\in\cplus,
\end{equation}
and $(\beta - A_{s,-1})^{-1}$ is a unitary operator from $\cH_{s,-1}$ onto $\cH_s$.
 
\item With $\cH_{s, -1}$ identified concretely as in \eqref{eq:Hs-1} and $\iota\,\Hscr_s$ identified with $\Hscr_s$, the control operator $B_s : \cU \to \cH_{s,-1}$ is
\begin{equation}\label{eq:Bs}
  B_su = (\mu,\mu_*)\mapsto \bbm{\varphi(\mu) \\1} u +\Zscr_s,\qquad u\in\Uscr,\,\mu,\mu_*\in\cplus.
\end{equation}
\end{enumerate}
\end{thm}

\begin{proof} The argument follows the proof of Thm.\ I.4.7, but we provide a more polished formulation. In order to establish that $\Hscr_{s,-1}$ is complete, take a Cauchy sequence $\sbm{z_{1,n}\\z_{2,n}}+\Zscr_s$ in $\cH_{s,-1}$. Then the Cauchy sequence  $R_\beta\sbm{z_{1,n}\\z_{2,n}}$ converges to some $\sbm{x_1\\x_2}$ in $\cH_s$. Solving $\sbm{x_1\\x_2}=R_\beta\sbm{z_1\\z_2}$ for $\sbm{z_1\\z_2}$, we obtain that $\sbm{z_{1,n}\\z_{2,n}}$ converges in $\Hscr_{s,-1}$ to
\begin{equation}\label{eq:alphaAdcalc}
\begin{aligned}
  \bbm{z_1\\z_2}+\Zscr_s = (\mu,\mu_*)\mapsto \bbm{(\beta - \mu)\, x_1(\mu) \\ (\beta + \mu_*)\, x_2(\mu_*)}
	+\Zscr_s, \qquad\mu,\mu_*\in\cplus. 
\end{aligned}
\end{equation}
Thus, $\Hscr_{s,-1}$ is a Hilbert space and $R_\beta\big|_{\Hscr_{s,-1}}$ clearly maps $\Hscr_{s,-1}$ unitarily onto $\Hscr_s$.

We next prove assertion one. Combining \eqref{eq:Hs-1norm} with \eqref{res}, we see that $\iota$ is continuous: for all $x\in\Hscr_s$ it holds that
$$
 \|\iota x\|_{\Hscr_{s,-1}} = \|R_\beta(x+\Zscr_s)\|_{\Hscr_s} \leq 
 \|(\beta-A_s)^{-1}\|_{\Lscr(\Hscr_s)}\,\|x\|_{\Hscr_s}.
$$
As $R_\beta\big|_{\Hscr_{s,-1}}$ is unitary from $\Hscr_{s,-1}$ onto $\Hscr_s$ and $\dom{A_s}=R_\beta\,\iota\,\Hscr_s$ is dense in $\Hscr_s=R_\beta\,\Hscr_{s,-1}$, it follows that $\iota\,\Hscr_s$ is dense in $\Hscr_{s,-1}$. 

For all $x\in\Hscr_s$ and $\alpha\in\cplus$,
$$
 	R_\alpha\,( x+\Zscr_s) =(\alpha-A_s)^{-1}x\in\dom{A_s};
$$
hence $\sbm{z_1\\z_2}+\Zscr_s=\iota\sbm{x_1\\x_2}$ only if $R_\alpha\sbm{z_1\\z_2}\in\dom{A_s}$, for all $\alpha\in\cplus$. Conversely, if there exists some $\alpha\in\cplus$ such that
$$
  \bbm{w_1\\w_2}:=R_\alpha\bbm{z_1\\z_2}=(\mu,\mu_*) \mapsto 
  \bbm{ \displaystyle \frac{z_1(\mu) - z_1(\alpha)}{\alpha - \mu} \\ 
    \displaystyle \frac{z_2(\mu_*) - \widetilde\varphi(\mu_*) \,z_1(\alpha)}{\alpha + \mu_*}} \in\dom{A_s}
$$
then by \eqref{eq:As}:
 $$
\begin{aligned}
 	\left((\alpha-A_s)\bbm{w_1\\w_2}\right)(\mu,\mu_*) &= 
  	\bbm{(\alpha-\mu)\, w_1(\mu)\\(\alpha+\mu_*)\,w_2(\mu_*)}
  		-\bbm{1\\\widetilde\varphi(\mu_*)}\lim_{\re\eta\to+\infty} \eta\,w_1(\eta),
\end{aligned}
$$
so that $\iota\,(\alpha-A_s)\sbm{w_1\\w_2}=\sbm{z_1\\z_2}+\Zscr_s$; please observe that $(\alpha-A_s)\sbm{w_1\\w_2}\in\Hscr_s$.

The decay condition picks out the unique representative in $\Hscr_s$ due to Cor.\ \ref{cor:xtendstozero}, and by writing $\iota^{-1}$ below, we mean the inverse of $\iota$ with domain $\iota\,\Hscr_s$. 

We next prove assertion two, and for this we temporarily denote the mapping in \eqref{eq:AsExt} by $\widetilde A_s$. In the beginning of the proof, we showed that
$$
	(\beta-\widetilde A_s)(x+\Zscr_s)=z+\Zscr_s\quad\iff\quad x=R_\beta\, z
$$
and $x$ lies in $\Hscr_s$ for  $z+\Zscr_s\in\Hscr_{s,-1}$, and so $x=R_\beta\,z$ if and only if $x+\Zscr_s=\iota\,R_\beta\,z$. In particular, $\beta-\widetilde A_s$ is injective:
$$
	(\beta-\widetilde A_s)(x+\Zscr_s)=\Zscr_s\quad\iff\quad x=R_\beta\,0=0,
$$
and we get $\iota\,R_\beta=(\beta-\widetilde A_s)^{-1}$. Comparing to \eqref{res}, we see that $(\beta-\widetilde A_s)^{-1}$ is the unique extension of the densely defined $\iota\,(\beta-A_s)^{-1}\iota^{-1}$ to an operator in $\Lscr(\Hscr_{s,-1})$; this implies that $\widetilde A_s$ is closed. Furthermore, $(\beta-\widetilde A_s)^{-1}=(\beta-A_{s,-1})^{-1}$, and inverting, we get  $A_{s,-1}=\widetilde A_s$. The above imply that
$$
	(\beta - A_{s,-1})^{-1} = \iota\, R_\beta\big|_{\Hscr_{s,-1}}
$$
and the identification $\iota\,\Hscr_s=\Hscr_s$ means that we may remove $\iota$ and $\iota^{-1}$ from the above formulas. Since $\beta$ is arbitrary in $\cplus$ the correctness of \eqref{eq:AsExtRes} follows.

It remains only to prove assertion three. By \eqref{eq:Hs-1norm} and \eqref{resB}, the operator $\widetilde B_s$ in \eqref{eq:Bs} maps into $\Hscr_{s,-1}$. Then \eqref{eq:AsExtRes} and \eqref{resB} give $\widetilde B_s=B_s$.
\end{proof}
 
Strictly speaking, $\Hscr_{s,-1}$ is not a RKHS, because its elements are \emph{equivalence classes} of functions rather than \emph{functions}. However, if we agree to represent an equivalence class in $\Hscr_{s,-1}$ by the function whose first component vanishes at the rigging point $\beta$ then we can obtain a reproducing kernel for $\Hscr_{s,-1}$. 

\begin{prop}
With $\beta$ the parameter used in the rigging, the space 
\begin{equation}\label{eq:Hs-1betadef}
	\Hscr_{s,-1}^\beta:=\set{\bbm{z_1\\z_2}\bigmid \bbm{z_1\\z_2}+\Zscr_s\in\Hscr_{s,-1},~ z_1(\beta)=0}
\end{equation}
with the norm \eqref{eq:Hs-1norm} is a Hilbert space of $\sbm{\Yscr\\\Uscr}$-valued analytic functions on $\cplus\times\cplus$. This space has the reproducing kernel 
\begin{equation}\label{eq:extrapolkernel}
\begin{aligned}
	&K_{s,-1}(\mu,\mu_*,\lambda,\lambda_*) := \\
	&\qquad \bbm{\beta-\mu&0\\0&\beta+\mu_*}
		K_s(\mu,\mu_*,\lambda,\lambda_*)
		\bbm{\overline\beta-\overline\lambda&0\\0&\overline\beta+\overline{\lambda_*}}.
\end{aligned}
\end{equation}
\end{prop}

The condition $x_1(\beta)=0$ picks out a unique representative of every equivalence class in the extrapolation space. We call $\Hscr_{s,-1}^\beta$ \emph{the $\beta$-normalized extrapolation space of $\SmallSysNode_s$}. By \eqref{eq:Hs-1betadef} the natural embedding of the state space $\Hscr_s$ into $\Hscr_{s,-1}^\beta$ is
\begin{equation}\label{eq:iotabeta}
	\iota^\beta \bbm{x_1\\x_2} := \bbm{x_1\\x_2}-\bbm{1\\\widetilde\varphi(\cdot)}x_1(\beta),
	\quad \bbm{x_1\\x_2}\in\Hscr_s.
\end{equation}

\begin{proof}
For $\sbm{x_1\\x_2}=R_\beta\sbm{z_1\\z_2}\in\Hscr_s$ with $z_1(\beta)=0$, we have from \eqref{eq:AsExt}:
$$
\begin{aligned}
	\Ipdp{\bbm{z_1(\lambda)\\z_2(\lambda_*)}}{\bbm{\gamma\\\nu}}_{\sbm{\Yscr\\\Uscr}}
		&= \Ipdp{\bbm{\beta-\lambda&0\\0&\beta+\lambda_*}e_s(\lambda,\lambda_*)\bbm{x_1\\x_2}}
			{\bbm{\gamma\\\nu}}_{\sbm{\Yscr\\\Uscr}} \\
	&=\Ipdp{\bbm{x_1\\x_2}}
		{e_s(\lambda,\lambda_*)^*
		\bbm{\overline\beta-\overline\lambda&0\\0&\overline\beta+\overline{\lambda_*}}
		\bbm{\gamma\\\nu}}_{\cH_s} \\
	&=\Ipdp{R_\beta\bbm{z_1\\z_2}}
		{R_\beta\,e_{s,-1}(\lambda,\lambda_*)^*\bbm{\gamma\\\nu}}_{\cH_s} \\
	&=\Ipdp{\bbm{z_1\\z_2}}{K_{s,-1}(\cdot,\blam)\bbm{\gamma\\\nu}}_{\cH^\beta_{s,-1}} ;
\end{aligned}
$$
thus \eqref{eq:extrapolkernel} works as the (unique) reproducing kernel of $\Hscr_{s,-1}^\beta$.
\end{proof}

\section{Recovering the unitary model for the disk}\label{sec:recovering}
 
As in \S I.6 we use the following function for mapping $\D$ one-to-one onto $\cplus$:
$$
  m_\alpha(z)=\frac{\alpha-\overline\alpha z}{1+z},\quad z\in\D,
$$
where the parameter $\alpha\in\cplus$ is arbitrary but fixed. The inverse of $m_\alpha$ is
$$
  m_\alpha^{-1}(\mu)=\frac{\alpha-\mu}{\overline\alpha+\mu},\quad\mu\in\cplus.
$$
Recalling from (I.6.5) that the Cayley transform with parameter $\alpha$ in (I.6.1) of a passive system node $\SmallSysNode$ has the transfer function $\phi_\alpha(z):=\varphi\big(m_\alpha(z)\big)$, we define a positive kernel function $\mathbf K_s(z,w)$ on $\D\times\D$ as in \eqref{eq:deBRkern} with $\phi$ replaced by $\phi_\alpha$:
\begin{equation}\label{eq:deBRkern2}
 \mathbf K_{s,\alpha}(z,w) :=
  \bbm{\displaystyle \frac{1-\phi_\alpha(z)\,\phi_\alpha(w)^*}{1-z\overline{w}}
    & \displaystyle\frac{\phi_\alpha(z)-\phi_\alpha(\overline w)}{z-\overline{w}} \\
   \displaystyle\frac{\phi_\alpha(\overline z)^*-\phi_\alpha(w)^*}{z-\overline{w}}
    & \displaystyle\frac{1-\phi_\alpha(\overline z)^*\,\phi_\alpha(\overline w)}{1-z\overline{w}}},
\end{equation}
$z,w\in\D$. The induced RKHS is denoted by $\mathbf H_{s,\alpha}$ and evaluation at $z\in\D$ in this space is denoted by $\mathbf e_{s,\alpha}(z)$.

\begin{lem}
The reproducing kernel in \eqref{eq:deBRkern2} can be written
\begin{equation}\label{eq:caylkernels}
  \mathbf K_{s,\alpha}(z,w)= \frac{1}{2\re\alpha}\bbm{\overline\alpha+\mu&0\\0&\alpha+\mu_*} K_s(\mu,\mu_*,\lambda,\lambda_*)
    \bbm{\alpha+\overline\lambda&0\\0&\overline\alpha+\overline{\lambda_*}},
\end{equation}
where $z,w\in\D$ and $\mu=m_\alpha(z)$, $\mu_*=m_{\overline\alpha}(z)$, $\lambda=m_\alpha(w)$, and $\lambda_*=m_{\overline\alpha}(w)$.

Moreover, the following mapping $\Xi_{s,\alpha}$ defines a unitary operator from $\mathbf H_{s,\alpha}$ onto $\Hscr_s$:
\begin{equation}\label{eq:XiDef}
\begin{aligned}
  \Xi_{s} \bbm{f\\g} :=& (\mu,\mu_*)\mapsto \sqrt{2\re\alpha}
    \bbm{\displaystyle\frac{f\big(m_\alpha^{-1}(\mu)\big)}{\overline\alpha+\mu}\\
      \displaystyle\frac{g\big(m_{\overline\alpha}^{-1}(\mu_*)\big)}{\alpha+\mu_*}},\\
    &\qquad \bbm{f\\g}\in\mathbf H_{s,\alpha},\,\mu,\mu_*\in\cplus.
\end{aligned}
\end{equation}
\end{lem}

Please pay attention to the complex conjugates on $\alpha$ in the formulas involving $\mu_*$ and $\lambda_*$.

\begin{proof}
The upper-left operator of \eqref{eq:caylkernels} is correct by (I.6.13) and the lower-right operator is correct due to (I.6.17). Furthermore,
$$
	z-\overline w = m_\alpha^{-1}(\mu)-\overline{m_{\overline\alpha}^{-1}(\lambda_*)}
		= \frac{(\overline{\lambda_*}-\mu) \, 2\re\alpha}
			{(\overline\alpha+\mu)(\overline\alpha+\overline{\lambda_*})}
$$
implies that the upper-right operator of \eqref{eq:caylkernels} is correct:
$$
	\frac{\phi_\alpha(z)-\phi_\alpha(\overline w)}{z-\overline w} = 
		\frac{1} 	{2\re\alpha} (\overline\alpha+\mu)
	\frac{\varphi(\mu)-\varphi(\overline{\lambda_*})}{\overline{\lambda_*}-\mu}	
		(\overline\alpha+\overline{\lambda_*}).
$$
Taking instead $z-\overline w=m_{\overline\alpha}^{-1}(\mu_*)-\overline{m_{\alpha}^{-1}(\lambda)}$, we obtain the lower-left corner of \eqref{eq:caylkernels}.

From \eqref{eq:caylkernels} and \eqref{eq:XiDef} one immediately obtains that
\begin{equation}\label{eq:XisKern}
	\Xi_{s} \, \mathbf e_{s,\alpha}(w)^* = \frac{1}{\sqrt{2\re\alpha}} \,e_s(\lambda,\lambda_*)^*
		\bbm{\alpha+\overline\lambda&0\\0&\overline\alpha+\overline{\lambda_*}}; 
\end{equation}
hence using \eqref{eq:caylkernels} in the first equality:
$$
\begin{aligned}
	\Ipdp{\Xi_{s} \, \mathbf e_{s,\alpha}(w)^*\bbm{\gamma\\\nu}}
		{\Xi_{s} \, \mathbf e_{s,\alpha}(z)^*\bbm{y\\u}}_{\Hscr_s} &= 
	\Ipdp{\mathbf K_{s,\alpha}(z,w) \bbm{\gamma\\\nu}}{\bbm{y\\u}}_{\sbm{\Yscr\\\Uscr}} \\
	&= \Ipdp{ \mathbf e_{s,\alpha}(w)^*\bbm{\gamma\\\nu}}
		{\mathbf e_{s,\alpha}(z)^*\bbm{y\\u}} _{\mathbf H_{s,\alpha}}
\end{aligned}
$$
for all $w,z\in\D$, $\gamma,y\in\Yscr$, and $\nu,u\in\Uscr$. Thus \eqref{eq:XiDef} can be extended by linearity and operator closure into an isometry from $\mathbf H_{s,\alpha}$ into $\Hscr_s$ and by \eqref{eq:XisKern} the range of this isometry contains 
$$
  \spn\set{e_s\big(m_\alpha(w),m_{\overline\alpha}(w)\big)^*\bbm{\gamma\\\nu}\bigmid w\in\D,\,\gamma\in\Yscr,\,\nu\in\Uscr}.
$$
Take $\sbm{f\\g}\in\Hscr$ perpendicular to this linear span. Setting $\nu=0$ and observing that $m_\alpha(\D)=\cplus$, we get $f=0$; then also $g=0$. Thus $\Xi_{s,\alpha}$ is unitary from $\mathbf H_{s,\alpha}$ to $\Hscr_s$. 
\end{proof}

The preceding lemma is in agreement with statements 2 of Propositions I.6.2 and I.6.3. Further, we have the following result:

\begin{prop}\label{prop:simpdbrrecover}
For all $\alpha\in\cplus$, the following claims are true:
\begin{enumerate}
\item The Cayley transform (I.6.1) with parameter $\alpha\in\cplus$ of $\SmallSysNode_s$ is the unitary operator $\sbm{\dA_{s,\alpha}&\dB_{s,\alpha}\\\dC_{s,\alpha}&\dD_{s,\alpha}}:\sbm{\Hscr_s\\\Uscr}\to\sbm{\Hscr_s\\\Yscr}$ given by
\begin{equation}\label{eq:simpcayley}
\begin{aligned}
  \dA_{s,\alpha}\bbm{h\\\ell} &= (\mu,\mu_*) \mapsto 
  \bbm{ \displaystyle \frac{\overline\alpha+\mu}{\alpha-\mu}\, h(\mu) 
    - \frac{2\re\alpha}{\alpha-\mu}\, h(\alpha)
  \\ \displaystyle \frac{\overline\alpha-\mu_*}{\alpha+\mu_*}\,\ell(\mu_*)-
    \frac{2\re\alpha}{\alpha+\mu_*}\,\widetilde\varphi(\mu_*) \, h(\alpha)}, \\
  \dB_{s,\alpha} \,u &= (\mu,\mu_*) \mapsto \sqrt{2\re\alpha}
  \bbm{\displaystyle \frac{\varphi(\mu)-\varphi(\alpha)}{\alpha - \mu}u
    \\ K_c(\mu_*,\overline\alpha) \,u}, \qquad \mu,\mu_*\in\cplus,\\ \displaystyle
  \dC_{s,\alpha} \bbm{h\\\ell} &= \sqrt{2\re\alpha}\,h(\alpha),\quad
  \dD_{s,\alpha} \, u = \varphi(\alpha)\, u,\qquad \bbm{h\\\ell}\in\Hscr_s,\, u\in\Uscr.
\end{aligned}
\end{equation}

\item The operator \eqref{eq:XiDef} implements a unitary similarity between $\sbm{\dA_s&\dB_s\\\dC_s&\dD_s}$ in \eqref{eq:simpcayley} and $\sbm{\mathbf A_s&\mathbf B_s\\\mathbf C_s&\mathbf D_s}$ in \eqref{eq:deBRunitary} built for $\phi(z):=\varphi\big(m_\alpha(z)\big)$, $z\in\D$:
\begin{equation}\label{eq:discrcontrint}
  \bbm{\dA_{s,\alpha}\,\Xi_{s}&\dB_{s,\alpha}\\\dC_{s,\alpha}\,\Xi_{s}&\dD_{s,\alpha}} =
  \bbm{\Xi_{s}\,\mathbf A_s&\Xi_{s}\,\mathbf B_s\\\mathbf C_s&\mathbf D_s}.
\end{equation}
\end{enumerate}
\end{prop}

\begin{proof}
Eq. (I.6.1) gives $\dD_{s,\alpha}=\varphi(\alpha)$ and for $\dA{s,\alpha}$, the formula in \eqref{eq:simpcayley} is immediate from Prop.\ \ref{P:resolA}. The formulas for $\dB_{s,\alpha}$ and $\dC_{s,\alpha}$ are \eqref{resB'} and \eqref{Cres'} renormalized. For the intertwinement, we obtain:
$$
\begin{aligned}
	\dC_{s,\alpha}\,\Xi_{s}\bbm{f\\g} &= 2\re\alpha
		\, \frac{f\big(m_\alpha^{-1}(\alpha)\big)}{2\re\alpha}
	= \mathbf C_s \bbm{f\\g},\\
	\big(\Xi_{s}\,\mathbf B_s\, u\big)(\mu,\mu_*) &= 
	\sqrt{2\re\alpha} \bbm{\displaystyle \frac{\varphi(\mu)-\varphi(\alpha)}{\frac{\alpha-\mu}
		{\overline\alpha+\mu}\,(\overline\alpha+\mu)} \\ K_c(\mu_*,\overline\alpha) }u=\mathbf B_{s,\alpha}\, u,
	\qquad\text{and} \\
	 \left(\mathbf A_{s,\alpha}\, \Xi_{s}\bbm{f\\g}\right)(\mu,\mu_*) &= \sqrt{2\re\alpha}
		\bbm{ \displaystyle \frac{f(z)-f(0)}{\alpha-\mu} \\
		\displaystyle  \frac{z\, g(z)-\widetilde\varphi(\mu_*)\,f(0)}{\alpha+\mu_*} } \\
	&= \left(\Xi_{s}\,\mathbf A_s \bbm{f\\g}\right)(\mu,\mu_*),
\end{aligned}
$$
where again $z=m_\alpha^{-1}(\mu)=m_{\overline\alpha}^{-1}(\mu_*)$ and $m_\alpha^{-1}(\alpha)=0$.

\end{proof}

We note that \eqref{eq:simpcayley} combines the Cayley transforms (I.6.6) and (I.6.18) of $\SmallSysNode_o$ and $\SmallSysNode_c$ in a way similar to how $\SmallSysNode_s$ in Thm.\ \ref{thm:consexplicit} combines (I.5.4--6) and (I.4.43). Also, $h(\alpha)$ in \eqref{eq:simpcayley} plays the role of $\tau_{c,\alpha}\,x=(\Gamma x)(\alpha)$ in (I.6.18).

\appendix

\section{Non-invertible intertwinement}\label{app:intertwinement}

This section contains results on intertwinements, which are not part of the main story. In this section, no assumptions on passivity are made.

The standard transfer function of a system node $\SmallSysNode$ with input space $\Uscr$, state space $\Xscr$, and output space $\Yscr$ only considers the input/output behavior of the system. We now extend the concept of transfer function in a way which also provides information on the state trajectory. Namely, we extend it into the mapping $\sbm{x_0\\\widehat u(\lambda)}\mapsto\sbm{\widehat x(\lambda)\\\widehat y(\lambda)}$, where $x_0$ is the initial state of a Laplace-transformable trajectory $(u,x,y)$ of the system and the hats denote the right-sided Laplace transforms. 

\begin{defn}
By the \emph{input/state/output (i/s/o) resolvent} of a system node with i/s/o spaces $(\Uscr,\Xscr,\Yscr)$, we mean the following family of bounded linear operators from $\sbm{\Xscr\\\Uscr}$ into $\sbm{\Xscr\\\Yscr}$:
\begin{equation}\label{eq:isores}
	\Sfrak(\lambda):=\bbm{(\lambda-A)^{-1}&(\lambda-A_{-1})^{-1}B \\ C(\lambda-A)^{-1} & \widehat\Dfrak(\lambda)},
		\qquad \lambda\in\res A.
\end{equation}
\end{defn}

It is immediate from Prop.\ I.3.10 that the i/s/o resolvent of the dual system $\SmallSysNode^*$ at $\lambda\in\res{A^*}$ is $\Sfrak^d(\lambda)=\Sfrak(\overline\lambda)^*$. Furthermore, a system node is uniquely determined by its i/s/o transfer function at any single point $\alpha\in\dom{\Sfrak}$: if $\Sfrak(\alpha)$ is determined then the following operator is also determined:
$$
	\bbm{A(\alpha-A)^{-1}&\alpha\,(\alpha-A_{-1})^{-1}B \\ C(\alpha-A)^{-1} & \widehat\Dfrak(\alpha)} = \SysNode
	\bbm{(\alpha-A)^{-1}&(\alpha-A_{-1})^{-1}B \\ 0 & 1},
$$
where the last operator maps $\sbm{\Xscr\\\Uscr}$ \emph{onto} $\dom{\SmallSysNode}$; see p.\ 740 in Part I.

In the present paper, the four component operators of $\Sfrak$ in fact play a more important role than the system-node components $A$, $B$, $C$ themselves, and much of the theory could be written in terms of these operators. However, here we choose a more explicit exposition which is more in line with the notation in, e.g., \cite{ADRSBook}.

\begin{lem}\label{lem:dualinter}
The following conditions are equivalent for two system nodes $\SmallSysNode_0$ and $\SmallSysNode_1$ with state spaces $\Xscr_0$ and $\Xscr_1$, respectively, and a bounded operator $E:\Xscr_0\to\Xscr_1$:
\begin{enumerate}
\item The operator $E$ intertwines $\SmallSysNode_0$ with $\SmallSysNode_1$.
\item The operator $E^*$ intertwines $\SmallSysNode_1^*$ with $\SmallSysNode_0^*$.
\item The following operator inclusion holds:
\begin{equation}\label{eq:intincl}
	\bbm{E&0\\0&1}\SysNode_0 \subset \SysNode_1\bbm{E&0\\0&1}, 
\end{equation}
where $\SmallSysNode_1\sbm{E&0\\0&1}$ is defined on its maximal domain
$$
	\set{\bbm{x\\u}\in\bbm{\Xscr_0\\\Uscr}\bigmid\bbm{Ex\\u}\in\dom{\SysNode_1}}.
$$
\item For all $\lambda\in\res{A_0}\cap\res{A_1}$ (which contains some right-half plane of $\C$), the operator $E$ intertwines the i/s/o resolvent of $\SmallSysNode_0$ with that of $\SmallSysNode_1$:
\begin{equation}\label{eq:isoresintertw}
\begin{aligned}
	&\bbm{E(\lambda-A_0)^{-1}&E(\lambda-A_{0,-1})^{-1}B_0 \\ C_0(\lambda-A_0)^{-1}  &\widehat\Dfrak_0(\lambda)} = \\
	&\qquad \bbm{(\lambda-A_1)^{-1}E&(\lambda-A_{1,-1})^{-1}B_1 \\ C_1(\lambda-A_1)^{-1}E  &\widehat\Dfrak_1(\lambda)}.
\end{aligned}
\end{equation}

\item  There exists one $\lambda\in\res{A_0}\cap\res{A_1}$ such that \eqref{eq:isoresintertw} holds.
\end{enumerate}
\end{lem}

\begin{proof}
By the definition of operator inclusion, \eqref{eq:intertwdom}--\eqref{eq:intertw} are equivalent to \eqref{eq:intincl}, i.e., claim one holds if and only if claim three holds. The following calculation shows that claim three implies claim two:
$$
\begin{aligned}
	\bbm{E&0\\0&1}^*\SysNode_1^* & \subset \left(\SysNode_1\bbm{E&0\\0&1}\right)^* \subset \left(\bbm{E&0\\0&1}\SysNode_0\right)^* \\
		&= \SysNode_0^*\bbm{E&0\\0&1}^*,
\end{aligned}
$$
where the first inclusion holds for all (unbounded) operators, the second inclusion follows from \eqref{eq:intincl}, and the equality holds because $\sbm{E&0\\0&1}$ is bounded; see \cite[Thm.\ 13.2]{RudinFA73}. This proves that statement one implies statement two, and applying this implication with $\SmallSysNode_{1-k}^*$ in place of $\SmallSysNode_k$ and $E^*$ in place of $E$, we obtain also that claim two implies claim one, since the (closed) system nodes and $E$ are all equal to their double adjoints.

In order to prove that statement one implies statement four, fix $\lambda\in\res {A_0}\cap\res{A_1}$ arbitrarily and assume \eqref{eq:intertwdom}--\eqref{eq:intertw}. Then it is easy to see that also the following two identities hold:
\begin{equation}\label{eq:CDint}
	\bbm{E&0\\0&1}\bbm{\bbm{1&0}\\\bbm{C_0\&D_0}} =
	 \bbm{\bbm{1&0}\\\bbm{C_1\&D_1}}\bbm{E&0\\0&1}\bigg|_{\dom{\SmallSysNode_0}}
\end{equation}
and
$$
	\bbm{E&0\\0&1}\bbm{\bbm{A_0\&B_0}\\\bbm{0&1}} =
	 \bbm{\bbm{A_1\&B_1}\\\bbm{0&1}}\bbm{E&0\\0&1}\bigg|_{\dom{\SmallSysNode_0}}.
$$
The latter of these implies that
\begin{equation}\label{eq:ABint}
	\bbm{\bbm{\lambda&0}-\bbm{A_1\&B_1}\\\bbm{0&1}}^{-1} \bbm{E&0\\0&1} =
	\bbm{E&0\\0&1}\bbm{\bbm{\lambda&0}-\bbm{A_0\&B_0}\\\bbm{0&1}}^{-1},
\end{equation}
since $\sbm{\sbm{\lambda&0}-\sbm{A_k\&B_k}\\\sbm{0&1}}^{-1}$ maps $\sbm{\Xscr_k\\\Uscr}$ into $\dom{\SmallSysNode_k}$; see the first three pages of \S I.3, by which we can also write \eqref{eq:isores} as
\begin{equation}\label{eq:SisoDec}
	\Sfrak(\lambda)=\bbm{\bbm{1&0} \\ \bbm{\CD}}
		\bbm{\bbm{\lambda&0}-\bbm{\AB} \\ \bbm{0&1}}^{-1}.
\end{equation}
Multiplying \eqref{eq:CDint} from the right by $\sbm{\sbm{\lambda&0}-\sbm{A_0\&B_0}\\\sbm{0&1}}^{-1}$, and using \eqref{eq:ABint} and \eqref{eq:SisoDec}, we get \eqref{eq:isoresintertw} for every $\lambda\in\res{A_0}\cap\res{A_1}$.

Statement four implies statement five because $\res{A_0}\cap\res{A_1}$ is nonempty. In order to prove that statement five implies statement one, we assume that $\lambda\in\res{A_0}\cap\res{A_1}$ is such that \eqref{eq:isoresintertw} holds. Then we claim that
\begin{equation}\label{eq:isoresintertemp}
	\bbm{E&0\\0&1}\bbm{\bbm{\lambda&0}-\bbm{A_0\&B_0} \\ \bbm{0&1}}^{-1} =
	\bbm{\bbm{\lambda&0}-\bbm{A_1\&B_1} \\ \bbm{0&1}}^{-1}\bbm{E&0\\0&1}.
\end{equation}
Indeed, the top half follows from \eqref{eq:isoresintertw} and the bottom half is trivial. Multiplying \eqref{eq:isoresintertemp} by $\sbm{\sbm{\lambda&0}-\sbm{A_0\&B_0} \\ \sbm{0&1}}$ from the right, we obtain
$$
	\bbm{E&0\\0&1}\dom{\SysNode_0}\subset\dom{\SysNode_1},
$$
and further multiplying by $\sbm{\sbm{\lambda&0}-\sbm{A_1\&B_1} \\ \sbm{0&1}}$ from the left, we get
\begin{equation}\label{eq:SecondIntertw}
\begin{aligned}
	&\bbm{\bbm{\lambda&0}-\bbm{A_1\&B_1} \\ \bbm{0&1}}\bbm{E&0\\0&1}\bigg|_{\dom{\SmallSysNode_0}} =	\\
	&\qquad \bbm{E&0\\0&1}\bbm{\bbm{\lambda&0}-\bbm{A_0\&B_0} \\ \bbm{0&1}}.
\end{aligned}
\end{equation}
Hence, \eqref{eq:intertwdom} and the top half of \eqref{eq:intertw} hold. Finally, we multiply the bottom half of the identity $\sbm{E&0\\0&1}\Sfrak_0(\lambda)=\Sfrak_1(\lambda)\sbm{E&0\\0&1}$, i.e,
$$
\begin{aligned}
	\bbm{C_0\&D_0}\bbm{\bbm{\lambda&0}-\bbm{A_0\&B_0} \\ \bbm{0&1}}^{-1} &=\\
	\bbm{C_1\&D_1}\bbm{\bbm{\lambda&0}-\bbm{A_1\&B_1} \\
		\bbm{0&1}}^{-1}\bbm{E&0\\0&1},
\end{aligned}
$$
from the right by $\sbm{\sbm{\lambda&0}-\sbm{A_0\&B_0}\\\sbm{0&1}}$ and use \eqref{eq:SecondIntertw}, which gives us the bottom half of \eqref{eq:intertw}.
\end{proof}

We have the following consequences:

\begin{thm}\label{thm:intertwprops}
Let $E$ intertwine $\SmallSysNode_0$ with $\SmallSysNode_1$. Then:
\begin{enumerate}

\item If $E$ is surjective then \eqref{eq:intertwdom} and \eqref{eq:intincl} hold with equality.

\item If $E$ is unitary then $\SmallSysNode_0$ is energy preserving, co-energy preserving, or conservative if and only if $\SmallSysNode_1$ has the same property.

\item Defining $E_{-1}^\alpha:=(\alpha-A_{1,-1})\,E\,(\alpha-A_{0,-1})^{-1}$ for $\alpha \in\res{A_0}\cap\res{A_1}$, we get 
\begin{equation}\label{eq:ExtrInt}
\begin{aligned}
 E^\alpha_{-1}\big|_{\Xscr_0} &= E,\qquad E_{-1}^\alpha A_{0,-1} = A_{1,-1}\,E, \qquad E_{-1}^\alpha B_0 = B_1, \\
 E (\alpha-A_{0,-1})^{-1} &= (\alpha-A_{1,-1})^{-1}E_{-1}^\alpha,\qquad \alpha \in\res{A_0}\cap\res{A_1}.
\end{aligned}
\end{equation}
The extrapolated intertwinement $E_{-1}^\alpha$ is surjective if $E$ is surjective.  

Moreover, if $\alpha=\beta$ (the rigging parameter) then $E_{-1}:=E_{-1}^\beta$ inherits the following properties from $E$: isometricity, co-isometricity, and unitarity.
\end{enumerate}
\end{thm}

\begin{proof}
Assertion one follows from \eqref{eq:isoresintertemp} and the surjectivity of $E$:
$$
\begin{aligned}
	\dom{\SmallSysNode_1} &= \bbm{\bbm{\lambda&0}-\bbm{A_1\&B_1} \\ \bbm{0&1}}^{-1}
		\bbm{\Xscr_1\\\Uscr} \\
	&= \bbm{\bbm{\lambda&0}-\bbm{A_1\&B_1} \\ \bbm{0&1}}^{-1}\bbm{E\Xscr_0\\\Uscr} \\
	&=\bbm{E&0\\0&1}\bbm{\bbm{\lambda&0}-\bbm{A_0\&B_0} \\ \bbm{0&1}}^{-1} \bbm{\Xscr_0\\\Uscr}
	\\ &=\bbm{E&0\\0&I}\dom{\SmallSysNode_0}.
\end{aligned}
$$

Now assume that $E$ is unitary, so that \eqref{eq:intincl} holds with equality. Then, for every $\sbm{x\\u}\in\dom{\SmallSysNode_1}$ and $\sbm{z\\y}=\SmallSysNode_1\sbm{x\\u}$, we have $\sbm{E^*x\\u}\in\dom{\SmallSysNode_0}$ and $\sbm{E^*z\\y}=\SmallSysNode_0\sbm{E^*x\\u}$. For $\SmallSysNode_0$ energy preserving, we have
$$
  2\re\Ipd{z}{x} = 2\re\Ipd{E^*z}{E^*x}=\|u\|^2-\|y\|^2;
$$
thus $\SmallSysNode_1$ inherits energy preservation from $\SmallSysNode_0$. The converse implication is obtained by swapping the roles of $\SmallSysNode_0$ and $\SmallSysNode_1$ and using $E^*$ for $E$. The same argument for  $\sbm{E^*&0\\0&1}\SmallSysNode_1^* = \SmallSysNode_0^*\sbm{E^*&0\\0&1}$ gives that $\SmallSysNode_0$ is co-energy preserving if and only if $\SmallSysNode_1$ is co-energy preserving. Hence, $\SmallSysNode_0$ is conservative if and only if $\SmallSysNode_1$ is conservative.

The second line of \eqref{eq:ExtrInt} is trivial by the definition of $E_{-1}^\alpha$, and taking $\alpha=0$ in $(\alpha-A_{1,-1})E = E_{-1}^\alpha(\alpha-A_{0,-1})$, we get $A_{1,-1}E=E_{-1}^\alpha A_{0,-1}$. Then, for $\alpha\neq 0$, we on the other hand get $E_{-1}^\alpha\big|_{\Xscr_0}=E$. Line 2 of \eqref{eq:ExtrInt} combined with the upper-right corner of \eqref{eq:isoresintertw} gives $E_{-1}^\alpha B_0 = B_1$. The rest of the assertion is immediate from the definition of $E_{-1}^\alpha$ and the unitarity of $\beta-A_{k,-1}:\Xscr_k\to\Xscr_{k,-1}$.
\end{proof}

\def\cprime{$'$} \def\cprime{$'$}
\providecommand{\bysame}{\leavevmode\hbox to3em{\hrulefill}\thinspace}
\providecommand{\MR}{\relax\ifhmode\unskip\space\fi MR }
\providecommand{\MRhref}[2]{%
  \href{http://www.ams.org/mathscinet-getitem?mr=#1}{#2}
}
\providecommand{\href}[2]{#2}

\end{document}